%% file: main.tex
\documentclass{article}
\input{pkg+shortcut}
\usepackage{graphicx} % Required for inserting images
\title{On the Local Structure and Approximation Stability of \\Block Isotropic Gaussian Fields}
\author[1]{Munki Jeong}
\author[2]{Alexander Strang\thanks{Corresponding author: \texttt{alexstrang@berkeley.edu}} }

\affil[1]{Electrical Engineering and Computer Science, Gwangju Institute of Science and Technology}
\affil[2]{Department of Statistics, University of California, Berkeley}
\date{}

\begin{document}
\maketitle
\begin{abstract}
Skew-symmetric functions are a class of functions defined on a product space $M \times M$ that are antisymmetric with respect to the order of their inputs. In \cite{richland2024sharedendpointcorrelationshierarchyrandom}, the authors proved that non-deterministic skew-symmetric Gaussian fields cannot be stationary or isotropic and proposed an alternative notion: stationarity (isotropy) in each component space. Our work focuses on local quadratic approximations of the associated Gaussian fields. Local quadratic approximations to random fields are random polynomials parametrized by a jointly sampled gradient vector and Hessian matrix. We characterize the distribution of the corresponding random vectors and random matrices. Then, we study the error in the quadratic approximation, which is also a Gaussian field. We investigate the error induced by the quadratic approximation in three senses: the pointwise error, the maximal error over an ellipsoidal region, and the worst-case error for multivariate Gaussian inputs at a given confidence level. Next, we explore the limiting behavior of the worst-case error as the distance between an expansion point and evaluation points approaches zero and infinity. Finally, we study how, as the input dimension increases, the variance of multivariate Gaussian distributions must be restricted to keep the worst-case error bound constant.

\vspace{0.05 in}

\end{abstract}

%\tableofcontents
\newpage
\input{Intro}

\input{Gaussian_Field}
\input{Local_Structure}
\input{Conclusion}

\section*{Conflict of Interest}
The authors declare that they have no conflicts of interest.

\printbibliography 

\newpage
\appendix
\input{appendix}
\end{document}

%% file: pkg+shortcut.tex
\usepackage[utf8]{inputenc}
\usepackage[backend = biber, sortcites = true, style = numeric]{biblatex}
\usepackage{float}
\usepackage[english]{babel}
\usepackage[autostyle, english = american]{csquotes}
\MakeOuterQuote{"}
\usepackage{authblk}

\usepackage{amsmath}
\usepackage{amsthm}
\usepackage{amssymb}

\newtheorem{theorem}{Theorem}
\newtheorem{lemma}[theorem]{Lemma}
\newtheorem{corollary}[theorem]{Corollary}
\newtheorem{proposition}[theorem]{Proposition}

\newtheorem{definition}{Definition}
\newtheorem{notation}[definition]{Notation}
\newtheorem*{remark}{Remark}

\newtheorem*{convention}{Convention}
\newtheorem{assumption}{Assumption}

\usepackage{graphicx}
\usepackage{subcaption} 
\usepackage{tikz}       % loads xcolor automatically
\usetikzlibrary{decorations.pathreplacing}
\usepackage[most]{tcolorbox}
\usepackage{todonotes}
\usepackage[top=0.7in,left=0.7in,right=0.7in,bottom=0.7in]{geometry}
\usepackage{listings}
\usepackage{appendix} 

\usepackage[ruled,vlined]{algorithm2e}

\usepackage[colorlinks=true, linkcolor=blue, citecolor=blue, urlcolor=blue]{hyperref}
\usepackage{cleveref}  

\newcommand{\indp}{\perp\!\!\!\perp}
\newcommand{\mc}[1]{\mathcal{#1}}
\newcommand{\mb}[1]{\mathbb{#1}}

\newcommand{\E}{\mathbb{E}}

\newcommand{\bb}[1]{\Big( #1 \Big)}

\newcommand{\cd}[1]{\Bigg \{ #1 \Bigg \}}

\newcommand{\dc}[1]{\bigg  [#1 \bigg  ]}

\DeclareMathOperator{\Cov}{Cov}

\newcommand{\tx}[1]{\text{#1}}

\newcommand{\all}{\forall}

\newcommand{\grad}{\nabla}

\newcommand{\der}{\partial}
\newcommand{\dl}{\delta}
\newcommand{\iid}{\overset{i.i.d.}{\sim}}
\newcommand{\tr}{\mathsf{T}}

\newcommand{\rdr}{\mathbb{R}^D \times \mathbb{R}^D}
\newcommand{\vo}{\Delta x^{(1)}}
\newcommand{\vt}{\Delta x^{(2)}}
\newcommand{\too}{\Delta y^{(1)}}
\newcommand{\ttt}{\Delta y^{(2)}}

\newcommand{\xo}{x^{(1)}}
\newcommand{\xt}{x^{(2)}}
\newcommand{\yo}{y^{(1)}}
\newcommand{\yt}{y^{(2)}}
\newcommand{\zz}{z^{*}}
\newcommand{\zo}{z^{(1)}}
\newcommand{\zt}{z^{(2)}}

\newcommand{\hoo}{H^{(1,1)}}
\newcommand{\hot}{H^{(1,2)}}
\newcommand{\dro}{\grad^{(1)}}
\newcommand{\drt}{\grad^{(2)}}
\newcommand{\droo}{\grad^{(1,1)}}
\newcommand{\drot}{\grad^{(1,2)}}

\RequirePackage{tikz} % For \foreach used for Orcid icon
\newcommand{\orcidicon}{\includegraphics[width=0.32cm]{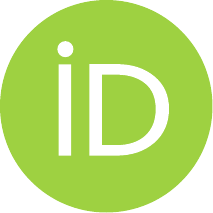}}

\foreach \x in {A, ..., Z}{%
\expandafter\xdef\csname orcid\x\endcsname{\noexpand\href{https://orcid.org/\csname orcidauthor\x\endcsname}{\noexpand\orcidicon}}
}

%% file: Intro.tex
\section{Introduction} 

Skew-symmetric functions are a class of functions defined on the product space $M \times M$, which are antisymmetric with respect to the order of their inputs. These functions generalize real skew-symmetric matrices, which can be viewed as functions defined on the product space of a finite set of natural numbers, to arbitrary index spaces. Skew-symmetric functions are common across the sciences. Notable applications include preference modeling and network flow analysis. 

For instance, in preference modeling, skew-symmetric functions represent the relative preference between two choices $A, B \in M$: the preference for choice $A$ over choice $B$ is the negative of the preference for $B$ over $A$ \cite{chu2005preference, benavoli2024tutoriallearningpreferenceschoices, chau2022learninginconsistentpreferencesgaussian}. Since \cite{chu2005preference} introduced a Bayesian nonparametric approach to preference learning via Gaussian processes, extensive research has been devoted to modeling and learning skew-symmetric preference functions. As another example, network flows are modeled as skew-symmetric functions on pairs of vertices $v_1, v_2 \in M$, where $f(v_1, v_2)$ denotes the flow from $v_1$ to $v_2$, and $f(v_2, v_1) = -f(v_1, v_2)$ denotes the flow in the opposite direction \cite{jiang2011statistical,strang2022network}. Gaussian fields are also employed here for stochastic modeling \cite{richland2024sharedendpointcorrelationshierarchyrandom}. More generally, Gaussian fields are a useful tool for modeling phenomena when the underlying mechanics are not well understood. This approach is also widely adopted in theoretical physics \cite{yamada2018hessian}.

%However, a single realization $f$ is just a collection of arbitrary values, since Gaussian fields only govern the ensemble of realizations. Consequently, any modeling approach based on a Gaussian field can only be described in a probabilistic way, such as in distribution, in expectation, and so on. On the other hand, in mathematical modeling, functions are often composed with other (deterministic or stochastic) functions, and specific spatial (e.g., geospatial modeling) or temporal (e.g., time-series modeling) aspects are ultimately examined. Probabilistic descriptions of these compositions are intractable and computationally demanding. To address these challenges, we study Taylor expansions of skew-symmetric functions sampled from Gaussian processes. These expansions yield random polynomials whose randomness is confined to their coefficients, thereby clearly decoupling randomness from spatial (or temporal) dependence. In particular, we focus on quadratic expansions. Quadratic expressions are widely used because they offer a tractable yet expressive means of capturing the rich local behavior of the original functions \cite{boyd2004convex}.

Many applied sciences study problems parameterized by functions, $f$. For example, most dynamical systems start from an ordinary differential equation model that is expressed $\frac{d}{dt} x(t) = f(x(t))$ for some function $f$. Commonly, these problems cannot be solved for a generic $f$. When the solution is continuous in $f$, $f$ is smooth, and inputs may be restricted to a small neighborhood, it is common to replace $f$ with a truncated expansion, $\tilde{f}$. In many applications, $\tilde{f}$ is a second-order Taylor approximation. The user usually proceeds by finding exact solutions to the approximate problem, substituting $\tilde{f}$ for $f$. Quadratic expressions are widely used because they offer a tractable yet expressive means of capturing the rich local behavior of the original functions \cite{boyd2004convex}. We are interested in the accuracy of local quadratic approximations to random models of skew-symmetric functions. For related work, see \cite{cebra2022similaritysuppressescyclicitysimilar,richland2024sharedendpointcorrelationshierarchyrandom}.

We introduce a specific subclass of skew-symmetric Gaussian processes, termed Block Isotropic Gaussian Processes. We characterize the local geometric structure of their derivatives, demonstrate that their quadratic approximations can be viewed as another Gaussian process parameterized by a random gradient and Hessian, and show that these are isotropic Gaussian vectors and Gaussian Orthogonal Ensembles. Finally, we analyze the approximation error between the original Gaussian process and its local quadratic approximation, providing the pointwise error distribution and extending this to a uniform error bound that holds with high probability over an elliptical region and for normally distributed random inputs.

The rest of the paper is organized as follows. In Section \ref{sec: Gaussian Fields}, we present the foundations of Gaussian random fields, including stationarity/isotropy, appropriate notions of derivatives, and several theorems characterizing the covariance of derivatives. This is followed by a discussion of random vectors and random matrices, and of their (un)correlation and independence. In Section \ref{sec: block isotropic}, we study skew-symmetric Gaussian fields and their properties, focusing in particular on the incompatibility of the general notion of stationarity (isotropy) with skew-symmetry in a product space. We then introduce an alternative notion, which we refer to as block stationarity (isotropy), and describe in detail the properties of block isotropic Gaussian fields. In Section \ref{sec: quadratic approximation}, we derive the quadratic Taylor approximation of block isotropic Gaussian fields at inputs of the form $[z,z]$ and characterize the quadratic approximation as a Gaussian field parametrized by Gaussian random vectors and matrices, with prescribed independence structure. Based on this structure, in Section \ref{sec: error}, we give a characterization of the discrepancy between the original Gaussian fields and their quadratic approximation as another Gaussian field. A simple closed-form expression for the covariance function of these Gaussian fields leads to a straightforward characterization of the pointwise error. We then extend the pointwise error result to the maximal error over an ellipsoidal region. We use this to construct a uniform upper bound on the error for random inputs sampled from a multivariate Gaussian distribution. Finally, we describe the limiting behavior of the bound near, and far, from the point of expansion. We conclude the paper by studying how an ellipsoidal region must concentrate about the point of expansion as the underlying dimension increases to maintain the same error characterization.

% \subsection{Contribution}
% \begin{enumerate}
%     \item  We characterize a specific subclass of skew-symmetric Gaussian Fields, whose local derivative structure can be characterized nicely, called Block Isotropic Gaussian Processes.
%      Leveraging the structure, the local quadratic approximations of a Block Isotropic Gaussian Process, which themselves turn out to be Gaussian processes, can be understood through the lens of a Gaussian random vector and well-known Gaussian Ensemble matrices.
%     \item We introduce the notion of an error process, defined as the discrepancy between the original block isotropic Gaussian processes and their local quadratic approximation. We formulate the pointwise error distribution and pointwise quantile function, and derive a compact formula via a dimensionality-reduction technique that leverages the intrinsic symmetry of isotropic Gaussian processes. We extend the pointwise bound to a uniform upper bound on the error that holds with high probability for cases where the inputs to the quadratic approximation are either drawn from an elliptical region centered about the point of expansion, or are normally distributed random vectors.
% \end{enumerate}

%% file: Gaussian_Field.tex
\section{Preliminaries on Gaussian Random Fields} \label{sec: Gaussian Fields}

\subsection{Gaussian Random Fields}
We present Gaussian fields as described in \cite{abrahamsen1997review} and \cite{adler2007random}.
\begin{definition}[Real-Valued Random Field]
    Let $(\Omega, \mc{F}, \mathbb{P})$ be a complete probability space and $M$ be a topological space (the index set). A real-valued random field is a collection of random variables $\{u(x)\}_{x \in M}$, or equivalently, a measurable mapping $u: \Omega \to \mb{R}^M$, where $\mb{R}^M$ denotes the space of all real-valued functions on $M$.
\end{definition}

In this work, we consider the case where the index set $M$ is a $2D$-dimensional Euclidean space, i.e., $M = \mb{R}^{2D}$. 
The random field is denoted by $u = \{u_x(\omega)\}_{x \in \mb{R}^{2D}}$. For a fixed $x \in \mb{R}^{2D}$, $u_x: \Omega \to \mb{R}$ is a real-valued random variable. We will primarily use the notation $u(x)$ rather than $u_x$ unless confusion arises.

\begin{definition} [Gaussian Random Field]
  A random field $\{u(x)\}_{x \in \mb{R}^{D}}$ is a Gaussian random field if for any finite set of points $\{x_1, \dots, x_n\} \subseteq \mathbb{R}^{D}$, every linear combination of $\big\{u(x_1), \dots, u(x_n) \big\}$, $\sum_{i=1}^n a_iu(x_i) $ follows a univariate Gaussian distribution.
\end{definition}

\begin{definition} [Equality in Distribution]
  Two Gaussian random fields $\{u(x)\}_{x \in \mathbb{R}^{D}}$ and $\{v(x)\}_{x \in \mathbb{R}^{D}}$ are said to be equal in distribution (or equal in law) if their finite-dimensional distributions are identical, i.e., for any finite set of points $\{x_1, \dots, x_n\} \subset \mathbb{R}^{D}$, the multivariate vectors $(u(x_1), \dots, u(x_n))$ and $(v(x_1), \dots, v(x_n))$ follow the same distribution:
  \begin{align*}
      \{u(x)\}_{x \in \mb{R}^{D}} \overset{d}{=} \{v(x)\}_{x \in \mathbb{R}^{D}}.
  \end{align*}
\end{definition}

Arbitrary random fields are not necessarily specified by their first two moments. However, Gaussian fields are entirely determined by their mean and covariance functions.
\begin{definition} [Mean and \text{Cov}ariance Function]
For a given Gaussian field $u$, the mean function $\mu: \mb{R}^{2D} \to \mb{R}$ and the covariance function $k_u: \mb{R}^{2D} \times \mb{R}^{2D} \to \mb{R}$ are defined as:
\begin{align*}
    \mu(x) &= \E[u(x)], \\
    k_u(x,y) &= \text{\text{Cov}}(u(x), u(y)) = \E[(u(x) - \mu(x))(u(y) - \mu(y))], \quad \all\text{ } x,y \in \mb{R}^{2D}.
\end{align*}
\end{definition}

We are predominantly interested in the case where the variance is finite. With this in mind, we consider Gaussian fields as elements of the Hilbert space $L^2(\Omega, M)$, or more precisely, we assume that for every fixed $x$, $u(x) \in L^2(\Omega, \mc{F}, \mb{P})$. This ensures that the covariance function $k_u$ is well-defined and finite everywhere.

Stationarity and isotropy are crucial properties because they imply that the random fields are invariant under translation and orthogonal coordinate transformations, respectively. Introducing symmetries simplifies analysis by reducing the number of free components of the model.

\begin{definition} [Stationary and Isotropic Gaussian Field]
A Gaussian field $u$ is \textbf{strictly stationary} if its finite-dimensional distributions are invariant under translations. 
That is $\{u(x)\}_{x \in \mb{R}^{D}} \overset{d}{=} \{u (x+a)\}_{x \in \mathbb{R}^{D}}$.
% That is, for any finite set of points $\{x_i\}_{i=1}^n \sst \mb{R}^{2D}$ and for any shift $a \in \mb{R}^{2D}$, the random vectors $\big( u(x_1), \ldots, u(x_n) \big)$ and $\big( u(x_1 + a), \ldots, u(x_n + a) \big)$ have the same distribution.

A Gaussian field $u$ is \textbf{weakly stationary} if its mean function is constant and its covariance function depends only on the difference vector $x-y$. i.e., there exists $\kappa: \mb{R}^{2D} \to \mb{R}$ such that:
\begin{align*}
    \mu(x) \equiv \mu \quad \text{and} \quad k_u(x,y) = \kappa(x - y).
\end{align*}

Further, a Gaussian field is \textbf{isotropic} if its mean is constant and its covariance depends only on the Euclidean distance $\|x - y\|$. i.e., there exists $\kappa_r: \mb{R}_{\geq 0} \to \mb{R}$ such that:
\begin{align*}
    k_u(x,y) = \kappa_r(\|x - y\|).
\end{align*}
\end{definition}

In general, strict stationarity implies weak stationarity, but the converse is not necessarily true. However, for Gaussian fields, the two notions are equivalent because the distribution is fully characterized by the first two moments. Thus, we use the term "stationary" without qualification. Notably, isotropic Gaussian fields are always stationary.

% \begin{convention}[Centered Gaussian Fields]
%   When dealing with stationary Gaussian fields, we will assume without loss of generality that the mean function is zero, i.e., $\mu(x) = 0$ for all $x \in \mb{R}^{2D}$.
% \end{convention}

\begin{notation}[Squared Distance Notation] \label{notation:kernel_function_h}
It is common to express the isotropic covariance function in terms of the squared distance. We posit that there exists a function $h: \mb{R}_{\geq 0} \to \mb{R}$ such that:
\begin{align}
    k_u(x,y) = h\big(\|x - y\|^2\big), \quad \all \text{ } x,y \in \mb{R}^{2D}.
\end{align}
For example, $h(\tau) = exp(-\tau/l^2)$ amounts to the squared exponential covariance function.
In this notation, the smoothness assumption $k_u \in C^{2n}(\mb{R}^{2D} \times \mb{R}^{2D})$ is equivalent to $h \in C^{n}(\mb{R}_{\geq 0})$.
\end{notation}

\begin{proposition} [Isotropic Gaussian Field is Orthogonally Invariant]
  Let $R: \mb{R}^{2D} \to \mb{R}^{2D}$ be an arbitrary orthogonal transformation (e.g., rotation or reflection).
  If $\{u(x)\}_{x \in \mb{R}^{2D}}$ is an isotropic Gaussian field, then it is distributionally invariant under $Q$, i.e., 
  \begin{align*}
      \{u(Qx)\}_{x \in \mb{R}^{2D}} \overset{d}{=} \{u(x)\}_{x \in \mb{R}^{2D}}.
  \end{align*}
Conversely, for a stationary Gaussian field, distributional invariance under all orthogonal transformations implies isotropy.
\end{proposition}

\begin{proof}
  Since Gaussian fields are entirely determined by their mean and covariance functions, it suffices to show that these are invariant.
  Let $u$ be isotropic. Since $\E[u(Qx)] = \mu = \E[u(x)]$, the mean is invariant.
  Furthermore, since $\|Qx - Qy\| = \|Q(x-y)\| = \|x-y\|$, we have:
  $$k_u(Qx, Qy) = \kappa_r(\|Qx - Qy\|) = \kappa_r(\|x - y\|) = k_u(x,y).$$
  Thus, the covariance is invariant.
  
  Conversely, assume $u$ is stationary and orthogonal invariant. Stationarity implies $k_u(x,y) = \kappa(x-y)$. Orthogonal invariance implies $\kappa(Q(x-y)) = \kappa(x-y)$ for any orthogonal $Q$. This is if and only if $\kappa$ is a function of $\| x-y\|$. Hence, $k_u$ depends only on $\|x-y\|$.
\end{proof}

\subsection{Differentiability of Gaussian Random Fields}

\begin{definition}[$L^2$ Continuity]
A Gaussian field $u$ is continuous in $L^2$ at $x \in \mb{R}^{2D}$ if for every sequence $\{x_n\}_{n=1}^{\infty}$ in $\mb{R}^{2D}$ such that $x_n \to x$, we have
\begin{align*}
    \lim_{n \to \infty} \E \bigg[ \big| u(x_n) - u(x) \big|^2 \bigg] = 0.
\end{align*}
\end{definition}

\begin{definition}[Directional Derivative]
Let $u: \mb{R}^{2D}$ be a deterministic function . The directional derivative of $u$ at $x$ in the direction $v \in \mb{R}^{2D}$, denoted by $D_v u(x)$, is defined as the limit of the quotient:
\begin{align*}
    D_v u(x) := \lim_{t \to 0}\frac{u(x + tv) - u(x)}{t}.
\end{align*}
In particular, when $v$ is the $i$-th standard basis vector $e_i$, the directional derivative reduces to the partial derivative $\partial_i u(x)$.
\end{definition}

\begin{definition} [Mean-Square Differentiability]
A Gaussian field $u$ is differentiable in $L^2$ (i.e., mean-square sense) at $x \in \mb{R}^{2D}$,  if there exists a random variable $D_v u(x) \in L^2(\Omega)$ such that
\begin{align*}
    \lim_{t \to 0} \E \bigg[ \bigg| \frac{u(x + tv) - u(x)}{t} - D_v u(x) \bigg|^2 \bigg]= 0 , \quad \text{ in every direction } v \in \mb{R}^{2D}.
\end{align*}
\end{definition}

\begin{proposition}[Interchange of Differentiation and Expectation]
    If the Gaussian field $u$ is differentiable in $L^2$ at $x \in \mathbb{R}^{2D}$, then we can interchange the order of differentiation and expectation, i.e.,
    \[
    \partial_i \mathbb{E}[u(x)] = \mathbb{E}[\partial_i u(x)], \quad 1 \le i \le 2D.
    \]
\end{proposition}

\begin{proof}
    Let $e_i$ denote the unit vector in the $i$-th direction. Recall that convergence in $L^2$ implies convergence in $L^1$. Therefore, the definition of mean-square differentiability implies:
    \[
    \lim_{t \to 0} \mathbb{E}\left| \frac{u(x + t e_i) - u(x)}{t} - \partial_i u(x) \right| = 0.
    \]
    By the linearity of expectation and Jensen's inequality, we have:
    \begin{align*}
    \left| \frac{\mathbb{E}[u(x + t e_i)] - \mathbb{E}[u(x)]}{t} - \mathbb{E}[\partial_i u(x)] \right| 
    &= \left| \mathbb{E}\left[ \frac{u(x + t e_i) - u(x)}{t} - \partial_i u(x) \right] \right| \\
    &\le \mathbb{E}\left| \frac{u(x + t e_i) - u(x)}{t} - \partial_i u(x) \right| \to 0.
    \end{align*}
    Thus, $\partial_i \E[u(x)] = \E[\partial_i u(x)]$.
\end{proof}

\begin{proposition} [Gaussianity under Linear Operations and Limits] \label{prop:Gaussianity_ops}
  Let $u$ be a Gaussian field. 
  \begin{enumerate}
      \item Any finite linear combination of values of $u$ is a Gaussian random variable.
      \item The $L^2$-limit of a sequence of Gaussian random variables is a Gaussian random variable.
  \end{enumerate}
\end{proposition}
\begin{proof}
  Part 1 follows directly from the definition of a Gaussian field (multivariate normality).
  
  For Part 2, let $\{X_n\}$ be a sequence of Gaussian random variables such that $X_n \sim \mathcal{N}(\mu_n, \sigma_n^2)$ and $X_n \to X$ in $L^2$. 
  
  First, recall that convergence in $L^2$ implies convergence of the first two moments.
  By the Cauchy-Schwarz inequality, $|\E[X_n] - \E[X]| = |\E[X_n - X]| \le \|X_n - X\|_{L^2} \to 0$, so $\mu_n \to \mu := \E[X]$.
  Similarly, by the reverse triangle inequality for norms, $\|X_n\|_{L^2} \to \|X\|_{L^2}$, which implies $\E[X_n^2] \to \E[X^2]$. Consequently, the variances converge: $\sigma_n^2 \to \sigma^2 := \text{Var}(X)$.
  
  Second, we recall that $X_n \to X$ in $L^2$ implies $X_n$ converges to $X$ in distribution (denoted $X_n \xrightarrow{d} X$).
  Consider the Cumulative Distribution Functions (CDF). Since each $X_n$ is Gaussian, its CDF is given by $F_n(x) = \Phi(\frac{x - \mu_n}{\sigma_n})$, where $\Phi$ is the standard normal CDF.
  Since $\mu_n \to \mu$ and $\sigma_n \to \sigma$, and $\Phi$ is a continuous function, we have pointwise convergence for all points where the limit CDF is continuous:
  \begin{align*}
      \lim_{n \to \infty} F_n(x) = \Phi\left(\frac{x - \mu}{\sigma}\right). 
  \end{align*} 
  Note that the right-hand side is the CDF of a Gaussian variable $\mathcal{N}(\mu, \sigma^2)$.
  Hence the random variable $X$ must follow the distribution $\mathcal{N}(\mu, \sigma^2)$. 
\end{proof}

\begin{proposition}[Derivative of Gaussian Field] \label{prop:Gaussanity_under_differential_operator}
  Let $u$ be a mean-square differentiable Gaussian field on $\mb{R}^{2D}$. Then the derivative $\partial_i u(x)$ is a Gaussian random variable.
\end{proposition}
\begin{proof}
    The derivative is defined as the $L^2$-limit of the quotient $\frac{1}{t}(u(x+te_i) - u(x))$. For any fixed $t$, this quotient is a linear combination of Gaussian variables, hence Gaussian. By Proposition \ref{prop:Gaussianity_ops}, the limit is also Gaussian.
\end{proof}

The following three results are presented in \cite{christakos1992random}:

\begin{proposition}[Sufficient Condition for Mean Square Differentiability]
  Let $u$ be a Gaussian field on $\mb{R}^{2D}$ with covariance function $k_u$ and a differentiable mean function. If the mixed partial derivatives $\frac{\der^2 k_u(x,y)}{\der x_i \der y_j}$ exist and are finite at the diagonal $x=y$ for all $1 \leq i,j \leq 2T$, then $u$ is mean-square differentiable at $x$.
  Furthermore, the covariance function of the derivatives is given by the second mixed partial derivatives of the kernel:
  \begin{align*}
   \Cov (\partial_i u(x), \partial_j u(y)) = \frac{\der^2 k_u(x,y)}{\der x_i \der y_j}.     
  \end{align*}
\end{proposition}

\begin{corollary}[Stationary Case]
  Let $u$ be a stationary Gaussian field with covariance $k_u(x,y) = \kappa(x-y)$. If $\kappa$ is twice differentiable at the origin $0$, then $u$ is mean-square differentiable everywhere. The covariance of the derivatives is given by:
  \begin{align*}
  \Cov(\partial_i u(x), \partial_j u(y)) = -\frac{\der^2 \kappa(x-y)}{\der x_i \der y_j}. 
  \end{align*}
\end{corollary}

\begin{proposition}[\text{Cov}ariance of Partial Derivatives] \label{prop:Covariance_of_partial_derivatives}
  If $L^2$ derivatives $\partial^\alpha_x u(x)$ and $\partial^\beta_y u(y)$ exist for multi-indices $\alpha, \beta$ at $x,y \in \mb{R}^{2D}$, then their covariance is given by differentiating the kernel:
  \begin{align*}
      \Cov \big(\partial^\alpha_x u(x), \partial^\beta_y u(y)\big) = \partial^\alpha_x \partial^\beta_y k_u(x,y).
  \end{align*}
\end{proposition}

\subsection{Preliminaries on Random Vectors and Random Matrices}
\begin{remark}
  Throughout this work, we consider the gradient and Hessian at $z \in \mb{R}^{D}$.
  While $u(z)$ is a scalar random variable, the gradient $\grad u(z)$ is a random vector in $\mb{R}^{D}$, and the Hessian $\grad^2 u(z)$ is a random matrix in $\mb{R}^{2D \times 2D}$. Rigorously, these are random variables taking values in  Hilbert spaces $L^2(\Omega; \mb{R}^{2D})$ and $L^2(\Omega; \mb{R}^{2D \times 2D})$ respectively.
  In particular, the inner product is defined as:
  \begin{align*}
      \langle X, Y \rangle_{L^2(\Omega; \mb{R}^{D})} &= \E[X^\tr Y], \quad X, Y \in L^2(\Omega; \mb{R}^{D}), \\
      \langle A, B \rangle_{L^2(\Omega; \mb{R}^{D \times D})} &= \E[\text{trace}(A^\tr B)], \quad A, B \in L^2(\Omega; \mb{R}^{D \times D}).
  \end{align*}
  Notably, the elements of Hilbert space $L^2(\Omega; \mb{R}^{D})$, $L^2(\Omega; \mb{R}^{D \times D})$ are Banach space-valued random variables where the norm is defined through the inner product and the metric is induced by the norm.
  With this understanding, real-valued, vector-valued, and matrix-valued random variables can be treated in a unified framework.
\end{remark}
Below are some important examples of random vectors and random matrices \cite{vershynin2009high, EdelmanRMT},
\begin{definition}[Gaussian random vectors, Isotropic Gaussian random vectors]
  A random vector $X = (X_1, \cdots, X_n)$ is multivariate Gaussian if for any constant vector $a \in \mb{R}^n$, the scalar random variable $a^{\tr} X = \sum_{i=1}^n a_i X_i$ follows a Gaussian distribution.
  When each component is uncorrelated and has the same variance, i.e., $\E \big[ (X-\E X)(X - \E X)^{\tr}\big] = \sigma^2I_n$, it is called an isotropic Gaussian random vector.
\end{definition}

\begin{proposition} \label{prop:orthogonal_invariance_of_gaussian_vector}
  A random vector $X \in \mb{R}^n$ is an isotropic Gaussian random vector if and only if for any orthogonal matrix $Q \in \mb{R}^{n \times n}$,
  \begin{align*}
      QX \overset{d}{=} X.
  \end{align*}
\end{proposition} The proof can be found in appendix \ref{appendix:orthogonal_invariance}.
\bigskip

For random matrices, we focus on the Gaussian Orthogonal Ensemble (GOE), which plays a central role in random matrix theory.
Similarly, we define the class "Gaussian Skew-symmetric Orthogonal Ensemble", which is not standard but useful in our analysis of the Hessian.
It should not be confused with the Gaussian Symplectic Ensemble (GSE), which is another important class of random matrices.

We adopt the element-wise sampling perspective definition of these ensembles, as introduced in \cite{EdelmanRMT}.
\begin{definition}[Gaussian Orthogonal Ensemble (GOE)] 
  Real matrix $M \in \mb{R}^{D \times D}$ is said to follow the Gaussian Orthogonal Ensemble, denoted by $M \sim \text{GOE}(D)$, if M is symmetric. Its entries $\{M_{ij}\}_{1 \leq i \leq j \leq D}$ are independent Gaussian random variables distributed as:
  \begin{align*}
      M_{ii} &\iid \mathcal{N}(0, 2), \quad 1 \leq i \leq D, \\
      M_{ij} &\iid \mathcal{N}(0, 1), \quad 1 \leq i < j \leq D.
  \end{align*}
\end{definition}

\begin{definition}[Gaussian Skew-Symmetric Orthogonal Ensemble (GSOE)]
  Real matrix $M \in \mb{R}^{D \times D}$ is said to follow the Gaussian Skew-Symmetric Ensemble, denoted by $M \sim \text{GSOE}(D)$, if M is skew-symmetric. Its entries $\{M_{ij}\}_{1 \leq i < j \leq D}$ are independent Gaussian random variables distributed as:
  \begin{align*}
      M_{ji} &= -M_{ij} \text{ with probabiilty one}, \quad 1 \leq i < j \leq D, \\
      M_{ij} &\iid \mathcal{N}(0, 1), \quad 1 \leq i < j \leq D, \\
      M_{ii} &= 0 \text{ with probabiilty one}, \quad 1 \leq i \leq D.
  \end{align*}
\end{definition}

\begin{theorem}
  Matrix $M \in \mb{R}^{D \times D}$ follows $\text{GOE}( D)$ if and only if there exists a matrix $A \in \mb{R}^{D \times D}$ with entries $A_{ij} \iid \mathcal{N}(0,1)$ such that
  \begin{align*}
    M \overset{d}{=} \frac{A + A^\tr}{\sqrt{2}}.
    \end{align*}
    Consequently, M is orthogonally invariant, i.e., for any orthogonal matrix $Q \in \mb{R}^{D \times D}$,
    \begin{align*}
        Q M Q^\tr \overset{d}{=} M.
    \end{align*}
\end{theorem}
\begin{proof}
The results on the representation of the GOE matrix are presented in \cite{EdelmanRMT}.
Orthogonal invariance results immediately from the invariance of the corresponding matrix $A$, which is proved in appendix \ref{appendix:orthogonal_invariance}.
\end{proof}
The analogous result holds for $\text{}$ as well.
\begin{corollary}
  Matrix $M \in \mb{R}^{D \times D}$ follows $\text{GSOE}(D)$ if and only if there exists a matrix $A \in \mb{R}^{D \times D}$ with entries $A_{ij} \iid \mathcal{N}(0,1)$ such that
  \begin{align*}
    M \overset{d}{=} \frac{A - A^\tr}{\sqrt{2}}.
    \end{align*}
    In particular, M is orthogonally invariant, i.e., for any orthogonal matrix $Q \in \mb{R}^{D \times D}$,
    \begin{align*}
        Q M Q^\tr \overset{d}{=} M.
    \end{align*}
\end{corollary}

\begin{proof}
Similar to the proof for the GOE case.
\end{proof}

Core concepts in probability theory can be defined in the same way for vector-valued random variables.
For example, \cite{dudley2002real} takes this way.

\begin{definition} [Independence of Banach Space Valued Random Variables]
  Let $X: \Omega \to B_1$ and $Y: \Omega \to B_2$ be random variables taking values in Banach spaces $B_1$ and $B_2$. $X$ and $Y$ are independent if for any Borel sets $S_1 \in \mc{B}(B_1)$ and $S_2 \in \mc{B}(B_2)$,
  \begin{align*}
      \mathbb{P}\big( X \in S_1, Y \in S_2 \big) = \mathbb{P}\big( X \in S_1 \big) \mathbb{P}\big( Y \in S_2 \big).
  \end{align*}
\end{definition}

% \begin{remark}
%   We refer to this simply as "independence," distinct from "free independence" found in non-commutative probability. Furthermore, we use the term "random matrix" to denote random variables taking values in the space of real matrices. Free probability theory is introduced in \cite{}
% \end{remark}

%%%%%%%%%%%%%%%%%%%%%%%%%%%%%%%%%%%%%%%%%%%%%%%%%%%%%%%%
% \begin{definition} [covariance in Hilbert Space $L^2(\Omega; \mb{R}^D)$]
%   For random vectors $X,Y \in L^2(\Omega; \mb{R}^D)$, the covariance is the inner product of the centered variables:
%   \begin{align*}
%       \text{Cov}(X,Y) = \E\Big[ \big( X - \E[X] \big) \big( Y - \E[Y] \big) \Big] = \bigg\langle X - \E[X], Y - \E[Y] \bigg\rangle_{L^2(\Omega; \mb{R}^D)}.
%   \end{align*}
% \end{definition}
% \begin{definition} [covariance Matrix in Hilbert Space $L^2(\Omega; \mb{R}^D)$]
% For random vectors $X,Y \in L^2(\Omega; \mb{R}^D)$, the covariance matrix $\Sigma_{XY}$ is the $2D \times 2D$ matrix defined as:
%  \begin{align*}
%     \Sigma_{XY} = \text{Cov}(X,Y) = \E\Big[ \big( X - \E[X] \big) \big( Y - \E[Y] \big)^D \Big].
% \end{align*}
% \end{definition}
%%%%%%%%%%%%%%%%%%%%%%%%%%%%%%%%%%%%%%%%%%%%%%%%%%%%%%%%

% \begin{proposition}[Correlation and Independence]
%   Let $X = X = (X_1, X_2, \cdots, X_k), Y= X = (Y_1, \cdots, Y_m )$ be vectors of jointly normal random variables. Then $X$ and $Y$ are independent if and only if they are uncorrelated, i.e., $\text{Cov}(X_i, Y_j) = 0$ for all components $i, j$.
% \end{proposition}

Furthermore, these concepts can be extended to matrix-valued random variables and their relationship to vector-valued random variables.
\begin{remark}[Borel Isomorphism]
The vectorization operation is an isometric isomorphism between $L^2(\Omega, \mathbb{R}^{D \times D})$ and $L^2(\Omega, \mathbb{R}^{D^2})$. Consequently, $L^2(\Omega, \mathbb{R}^{D \times D})$ and $L^2(\Omega, \mathbb{R}^{D^2})$ are equivalent up to Borel isomorphism.
\end{remark}

%\newpage
\section{Block Isotropic Gaussian Fields} \label{sec: block isotropic}
Hereafter, the Euclidean space $\mb{R}^{2D}$ should be understood as the product space $\rdr$.
\begin{notation} [Coupled Vectors in $\mb{R}^{2D}$]
  For $x \in \mb{R}^{D \times D}$, we denote the first $D$ components and the last $D$ components of $x$ by $x^{(1)} \in \mb{R}^D$ and $x^{(2)} \in \mb{R}^D$, respectively. That is,
  \begin{align}
      x &= (x_1, x_2, \ldots, x_D, x_{D+1}, x_{D+2}, \ldots, x_{2D}) \in \mb{R}^{2D} \label{not:allinone_vector}\\
      &= (x^{(1)}, x^{(2)}) \in \rdr \label{concatenated_vector}.
  \end{align}
  We use the notations (\ref{not:allinone_vector}) and \eqref{concatenated_vector} interchangeably. Regarding the indexing, we use $x_i$ to denote the $i^{th}$ component of $x \in \mb{R}^{2D}$, while we use $x^{(1)}_i$ and $x^{(2)}_i$ to denote the $i^{th}$ components of $x^{(1)} \in \mb{R}^D$ and $x^{(2)} \in \mb{R}^D$, respectively. 
\end{notation}

\begin{definition}[Flip operation] \label{def:flip_operation}
  We define the linear transformation  $(\cdot)^{flip}: \rdr \to \rdr$ as
  \begin{align*}
      x^{flip} = (\xo, \xt)^{flip} = (\xt, \xo)
  \end{align*}
\end{definition}

\subsection{Skew-Symmetric Gaussian Fields}

\begin{definition}[Skew-Symmetric Function]
  A function $f: \rdr \to \mb{R}$ is skew-symmetric if for any $x = (x^{(1)}, x^{(2)}) \in \mb{R}^D \times \mb{R}^D$,
  \begin{align*}
      f(x^{flip}) = - f(x)
  \end{align*}
\end{definition}

\begin{proposition}[Sign of Derivatives of Skew-Symmetric Functions] \label{prop:sign_of_derivatives_of_skew_symmetric_functions}
  Let $f: \rdr \to \mb{R}$ be a skew-symmetric function, and let the required derivatives exist. Then, for any multi-index $\alpha = (\alpha_1, \alpha_2, \ldots, \alpha_{D})$ and $\beta = (\beta_1, \beta_2, \ldots, \beta_{D})$, the following holds:
  \begin{align*}
      \partial^{\alpha}_{x^{(1)}} \partial^{\beta}_{x^{(2)}} f(x^{(2)}, x^{(1)}) = (-1)^{|\alpha| + |\beta| + 1} \partial^{\beta}_{x^{(1)}} \partial^{\alpha}_{x^{(2)}} f(x^{(1)}, x^{(2)}),
  \end{align*}
\end{proposition}
The proof of this proposition can be found in \cite{cebra2022similaritysuppressescyclicitysimilar}.
Notably, all skew-symmetric functions equal zero on the diagonal set $\{(x,x) \in \rdr: x \in \mb{R}^D\}$, since $f(x,x) = - f(x,x)$ implies $f(x,x) = 0$

The characterization of the skew-symmetric Gaussian field is presented in the previous work \cite{richland2024sharedendpointcorrelationshierarchyrandom}. We explain some results on which our work is based.
\begin{definition}[Skew-Symmetric Gaussian Field] 
  A Gaussian field $GF(\mu_f, k_f)$ on $\rdr$ is skew-symmetric if $\mu_f$ is $f(x^{flip}) = -f(x)$ for all $x \in \rdr$ with probability one.
\end{definition}

\begin{lemma}
  $GF(\mu_f, k_f)$ is a skew-symmetric Gaussian field if and only if its mean function $\mu_f$ and covariance function $k_f$ satisfy the following conditions:
  \begin{align*}
      \mu_f(x^{flip}) &= - \mu_f(x), \quad \all \text{ } x \in \rdr, \\
      k_f(x, y^{flip}) &=  - k_f(x, y), \quad \all \text{ } x,y \in \rdr.
  \end{align*}
\end{lemma}

\begin{lemma} (Properties of Non-Deterministic Skew-Symmetric Gaussian Field) \label{lemma:properties_of_skew_symmetric_gaussian_field}
  Let $GF(\mu_f, k_f)$ be a skew-symmetric Gaussian field $\mb{P} \bigg[ f(x) \neq \mu_f(x) \bigg] >0$. Then it follows that:

\begin{enumerate}
  \item $k_f: \rdr \times \rdr \to \mb{R}$ is neither strictly positive nor strictly negative. 
  \item $GF(\mu_f, k_f)$ is non-stationary.
\end{enumerate}
\end{lemma}

\begin{theorem} [Representation of Skew-Symmetric Gaussian Field] \label{thm:representation_of_skew_symmetric_gaussian_field} 
  Let $f$ be a Gaussian field on $\rdr$. Then, $f$ is skew-symmetric if and only if there exists a base Gaussian field $u$ on $\rdr$ such that
  \begin{align*}
      f(x) = u(x) - u(x^{flip}), \quad \all \text{ } x \in \rdr.
  \end{align*}
\end{theorem}
For example, in preference modeling and game theory, it is common to study skew-symmetric preference, or payout, functions that are constructed as a difference in utility functions \cite{chu2005preference, benavoli2024tutoriallearningpreferenceschoices, chau2022learninginconsistentpreferencesgaussian}.  Theorem 14 allows us to build skew-symmetric Gaussian fields from Gaussian fields without skew-symmetric constraints, referred to as "base" Gaussian fields.
\begin{definition} [Base Gaussian Field]
  Let $f$ be a skew-symmetric Gaussian field. Then, the Gaussian field $u$ defined as in Theorem \ref{thm:representation_of_skew_symmetric_gaussian_field} is the base Gaussian field associated with $f$.
\end{definition}

\subsection{Block Isotropic Gaussian Fields}
While stationarity and isotropy are desirable characteristics for analysis,
Lemma \ref{lemma:properties_of_skew_symmetric_gaussian_field} implies that every skew-symmetric non-deterministic Gaussian field is not stationary, so cannot be isotropic either. Nevertheless, analogous notions of stationarity and isotropy are introduced in \cite{richland2024sharedendpointcorrelationshierarchyrandom}. Recall that a skew-symmetric function is defined on the product space $\rdr$. Seeking distributional invariance for each component space separately is reasonable since the inputs $x, y$ to $f(x, y)$ represent elements of a shared component space $\mathbb{R}^D$ in most applications. For example, if $f$ is a preference function intended to compare elements drawn from the same space, it is natural to interpret each component space $\mathbb{R}^D$ in the same way and expect them to have the same structure \cite{benavoli2024tutoriallearningpreferenceschoices}. Consequently, we consider the invariance under translations and orthogonal transformations that are applied to each component space $\mb{R}^D$ separately, i.e., $\all \text{ } a^*\in \mb{R}^D \tx{ and orthogonal matrix } Q^* \in \rdr$,

\begin{align}
  &\{f(\xo, \xt)\}_{x\in \rdr} \overset{d}{=} \{f(\xo+ a^*, \xt + a^*)\}_{x\in \rdr} \\
  &\{f(\xo, \xt)\}_{x\in \rdr} \overset{d}{=} \{f(Q^*\xo, Q^*\xt)\}_{x\in \rdr} \label{eq:invariance_version_block_statinoary}
\end{align}
instead of invariance under arbitrary translations $T_a$ and orthogonal transformations $T_Q$ defined on the product space $\rdr$, 
\begin{align*}
  &\{f(\xo, \xt)\}_{x\in \rdr} \overset{d}{=} \{f(T_a(\xo, \xt)) \}_{x\in \rdr} \\
  &\{f(\xo, \xt)\}_{x\in \rdr} \overset{d}{=} \{f(T_Q(\xo, \xt))\}_{x\in \rdr} 
\end{align*}
  
%\begin{theorem} [Representation of Skew-Symmetric Gaussian Field] \label{thm:representation_of_skew_symmetric_gaussian_field} 
%  Let $f$ be a Gaussian field on $\rdr$. Then, $f$ is skew-symmetric if and only if there exists a base Gaussian field $u$ on $\rdr$ such that
 % \begin{align*}
 %     f(x^{(1)}, x^{(2)}) = u(x^{(1)}, x^{(2)}) - u(x^{(2)}, x^{(1)}), \quad \all \text{ } (x^{(1)}, x^{(2)}) \in \rdr.
%  \end{align*}
%\end{theorem}
%\begin{definition} [Base Gaussian Field]
%  Let $f$ be a skew-symmetric Gaussian field. Then, the Gaussian field $u$ defined as in Theorem \ref{thm:representation_of_skew_symmetric_gaussian_field} is the base Gaussian field associated with $f$.
%\end{definition}

%The notion of component-wise invariance can be formalized with respect to the associated base field.

A stochastic process exhibits block isotropy if it is defined on a product space, and its ensemble is invariant to unitary coordinate changes applied to the component spaces. If $f$ is expressed as a difference of base fields $u$, and $u$ is isotropic, then $f$ is block isotropic. Since we will work exclusively with this parameterization, we will, for brevity, use block isotropic to refer to models expressed as a difference of isotropic base fields.

\begin{definition}[Block Stationary/Isotropic Gaussian Field] \label{definition:block_isotropic_gaussian_field}
  The field $f$ is block stationary and block isotropic if the base field $u$ is stationary and isotropic, respectively.
\end{definition}

\begin{remark}
    Definition \ref{definition:block_isotropic_gaussian_field} implies the desired invariance, but there may exist processes with the desired invariance that cannot be expressed as a difference in isotropic base fields. 
\end{remark}

% In this regard, we introduce the weaker notions of invariance, which we call \textbf{block stationarity} and \textbf{block isotropy}.

% \begin{proof}
%   When $f$ is defined as above, it is straightforward to show that $f$ is skew-symmetric.
%   Conversely, let a skew-symmetric Gaussian field $f$ be given. We define the base field as follows:
%   \begin{align*}
%     &u(x) = 1/2  \cdot f(x). \label{eq:basefield}\\
%     &\text{ Then it follows that } u(x) - u(x^{flip}) = 1/2 \bigg(f(x) - f(x^{flip})\bigg) = 1/2 \bigg(f(x) + f(x)\bigg) = f(x). \nonumber
%   \end{align*}
% \end{proof}

\begin{convention} \label{convention:mean_zero_base_field}
  If the mean function of the base process $u$ is constant, then the mean function of the skew-symmetric Gaussian field $f$ is always zero. We will assume that the mean function of the stationary or isotropic base process $u$ is zero.
\end{convention}

Hereafter, we primarily focus on block-isotropic Gaussian fields. The covariance functions of block isotropic Gaussian fields can be expressed in terms of those of their corresponding base fields. This is introduced in \cite{richland2024sharedendpointcorrelationshierarchyrandom, benavoli2024tutoriallearningpreferenceschoices}
\begin{lemma}
  Let $f \sim GF(\mu_f, k_f)$ be a Block Isotropic Gaussian field with associated base field $u$ defined as in \cref{thm:representation_of_skew_symmetric_gaussian_field}, with covariance function $k_u$.
   Then, 
   \begin{align*}
    k_f(x,y) &= 2 \Big( k_u(x,y) - k_u\big(x,y^{flip}\big) \Big) \text{ for all }  x,y \in \rdr.
   \end{align*}
\end{lemma}

\begin{proof}
%This one is worth repeating since we will study these functions in detail.
  By the definition of $f$ and the linearity of expectation, we have
  \begin{align*}
      k_f(x,y) &= \E\Big[ f(x) f(y) \Big] \\
      &= \E\Big[ \big( u(x) - u(x^{flip}) \big) \big( u(y) - u(y^{flip}) \big) \Big] \\
      &= \E\Big[ u(x)u(y) - u(x)u(y^{flip}) - u(x^{flip})u(y) + u(x^{flip})u(y^{flip}) \Big] \\
      &= k_u(x,y) - k_u\big(x,y^{flip}\big) - k_u\big(x^{flip},y\big) + k_u\big(x^{flip},y^{flip}\big).
  \end{align*}
  Since $u$ is isotropic, $k_u\big(x^{flip},y^{flip}\big) = k_u(x,y)$ and $k_u\big(x^{flip},y\big) = k_u\big(x,y^{flip}\big)$.
  Then,
  \begin{align*}
      k_f(x,y) &= 2 \Big( k_u(x,y) - k_u\big(x,y^{flip}\big) \Big).
  \end{align*}
\end{proof}
\begin{remark}
  The $L^2$ smoothness class of $k_f$ is determined by that of $k_u$ and vice versa, since they are equivalent up to linear operations.
\end{remark}

\begin{definition}[Block-Orthogonal Transformation $Q^b$]
For an orthogonal transformation $Q: \mb{R}^D \to \mb{R}^D$, we define the block-orthogonal transformation $Q^b: \rdr \to \rdr$ as
\begin{align*}
    Q^b x = \begin{bmatrix}
    Q & 0 \\
    0 & Q
    \end{bmatrix} \begin{bmatrix} x^{(1)} \\ x^{(2)} \end{bmatrix} = \begin{bmatrix} Q x^{(1)} \\ Q x^{(2)} \end{bmatrix}, \quad \all \text{ } x = \begin{bmatrix} x^{(1)} \\ x^{(2)} \end{bmatrix} \in \rdr.      
\end{align*}
\end{definition}

In this way, it transforms $x^{\text{flip}} = \begin{bmatrix} x^{(2)} \\ x^{(1)} \end{bmatrix} \in \rdr$ into 
\begin{align*}
  Q^b x^{\text{flip}} = \begin{bmatrix} Qx^{(2)} \\ Qx^{(1)} \end{bmatrix} = (Q^bx)^{\text{flip}} \label{equation:flip_commute_block_orthogonal_transformation}.
\end{align*}
\\
\begin{proposition} [Block Isotropic Gaussian Field is Block-Orthogonally Invariant] \label{proposition:orthogoanal_invariance_of_f}
  Let $Q: \mb{R}^D \to \mb{R}^D$ be an arbitrary orthogonal transformation and let $Q^b: \rdr \to \rdr$ be the corresponding block-orthogonal transformation. If $\{f(x)\}_{x \in \rdr}$ is a block isotropic Gaussian field, then it is distributionally invariant under $Q^b$, i.e., 
  \begin{align*}
      \{f(Q^b x)\}_{x \in \rdr} \overset{d}{=} \{f(x)\}_{x \in \rdr}.
  \end{align*}
\end{proposition}
\begin{proof}
  Since mean-zero Gaussian fields are completely determined by their covariance functions, the conclusion immediately follows from the equality below.
  \begin{align*}
      k_f(Q^b x, Q^b y) &= 2 \Big( k_u(Q^b x, Q^b y) - k_u\big(Q^b x,(Q^b y)^{flip}\big) \Big) \\
      &= 2 \Big( k_u(Q^b x, Q^b y) - k_u\big(Q^b x,Q^b y^{flip}\big) \Big) \\ & = 2 \Big( k_u(x, y) - k_u\big(x,y^{flip}\big) \Big) \\ &
      = k_f(x,y),
  \end{align*} 
\end{proof}

%% file: Local_Structure.tex
\section{Local Quadratic Approximation of Block Isotropic Gaussian Processes} \label{sec: quadratic approximation}

The values of random fields are determined by both random and spatial effects. Expressing a random function $f_x(\omega)$, where $x \in \rdr$ and $\omega \in \Omega$, in terms of a Taylor basis $\{T_k : \rdr \to \mb{R}\}_{k=1}^{\infty}$ with corresponding random coefficients $\{A_k: \Omega \to \mb{R}\}_{k=1}^{\infty}$ separates these factors:
\begin{align*}
f_x(\omega) = \sum_{k=1}^{\infty} A_{k}(\omega)  T_k(x)
\end{align*}
This representation offers several advantages. First, it provides a clearer intuition, as the random coefficients can be interpreted as amplitudes associated with each basis function. Second, it reduces the study of infinite-dimensional objects to the analysis of countable (or finite) collections of coefficients.

In this work, we focus on the second-order Taylor expansion, as it balances expressivity and tractability \cite{boyd2004convex}. A quadratic local approximation can capture both local slopes and curvature. Furthermore, the coefficients corresponding to the bases $T_1$ and $T_2$ are given by the gradient and Hessian, respectively, making the quadratic expansion both contextually interpretable and technically tractable. We omit the dependence on the point about which we perform a Taylor expansion for brevity, unless clarity requires it.

%\pagebreak 
\subsection{Second-Order Taylor Approximation of Block Isotropic Gaussian Fields at Matched Inputs} \label{sec:Taylor_approximation}

The quadratic Taylor approximation at the point $z \in \rdr$, where the components of each space coincide, reveals interesting local derivative structures of block isotropic Gaussian fields. Consequently, the local quadratic approximation at these points can be characterized in a particularly elegant way.

\begin{notation}[Function value, gradient, and Hessian] \label{notation_gradient_Hessian}
    Let $f \sim GF(\mu_f, k_f)$ be a Gaussian field over $\rdr$.
    When evaluated at a given point $z = (\zo, \zo)$, the function value $f(z,z)$, gradient $\grad f(z,z)$, and Hessian $H(z,z)$ are a scalar, $2T$ dimensional vector, and  $2D \times 2D$ matrix respectively. We introduce the following notation for the entries and sub-components of gradients and Hessians.
    \begin{align*}
        \der^{(1)}_i f(\zo, \zt) &= \der_{\xo_i}f(\xo, \zt)|_{\xo = \zo} && \der^{(2)}_j f(\zo, \zt) = \der_{\xt_j}f(\zo, \xt)|_{\xt = \zt} \\
        \der^{(1, 1)}_{(i,j)} f(\zo, \zt) &= \der^2_{\xo_i \xo_j}f(\xo, \xt)|_{\xo = \zo} && \der^{(2,2)}_{(i,j)} f(\zo, \zt) = \der^2_{\xt_i \xt_j}f(\zo, \xt)|_{\xt = \zt} \\        
        \der^{(1, 2)}_{(i,j)} f(\zo, \zt) &= \der^2_{\xo_i \xt_j}f(\xo, \xt)|_{\substack{\xo = \zo \\ \xt= \zt}} && \der^{(2,1)}_{(i,j)} f(\zo, \zt) = \der^2_{\xt_i \xo_j}f(\xo, \xt)|_{\substack{\xo = \zo \\ \xt= \zt}}       
    \end{align*}

    \begin{align*}
        \tx{(Truncated Gradient)} \ g(\zo,\zt) &= [\grad f(\zo, \zt)]_{1:D}, &&\grad^{(2)}f(\zo,\zt) = [\grad f(\zo, \zt)]_{D+1:2D}\\
        \tx{(Diagonal Block)}\ H^{(1, 1)}(\zo,\zt) &= [H(\zo,\zt)]_{1:D, 1:D}  &&H^{(2, 2)}(\zo,\zt) = [H(\zo,\zt)]_{D+1:2D, D+1:2D} \\
        \tx{(Off-Diagonal Block)} \ H^{(1,2)}(\zo,\zt) &= [H(\zo,\zt)]_{1:D, D+1:2D} &&H^{(1,2)}(\zo,\zt) = [H(\zo,\zt)]_{D+1:2D, 1: D}
    \end{align*}
\end{notation}

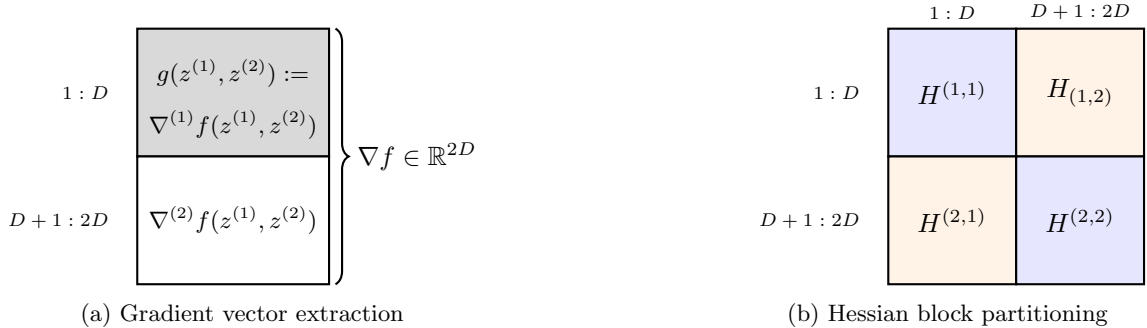
\begin{figure}[h]
    \centering
    % --- Figure (a): Gradient Vector Extraction ---
    \begin{subfigure}[b]{0.48\textwidth}
        \centering
        \begin{tikzpicture}[scale=0.85]
            
            % --- The Vector Blocks ---
            % Top half (Blue) - Increased width to 3.0 to fit text
            \draw[thick, fill=gray!30] (0,2) rectangle (3.0,4);
            \node at (1.5, 3.3) {\small $g (z^{(1)}, z^{(2)}):=$};
            \node at (1.5, 2.5) {\small $\nabla^{(1)} f(z^{(1)}, z^{(2)})$};
            
            % Bottom half (White)
            \draw[thick, fill=white] (0,0) rectangle (3.0,2);
            \node at (1.5, 1) {\small $\nabla^{(2)} f(z^{(1)}, z^{(2)})$};
            
            % --- Indices (Left) ---
            % Moved x from 0 to -0.3 to prevent overlap with the box
            \node[anchor=east] at (-0.3, 3) {\scriptsize $1:D$};
            \node[anchor=east] at (-0.3, 1) {\scriptsize $D+1:2D$};
            
            % --- Brace and Label (Right) ---
            % Brace moved to x=3.1 to match new box width
            \draw[decorate,decoration={brace,mirror,amplitude=5pt}, thick] (3.1,0) -- (3.1,4);
            \node[anchor=west] at (3.3, 2) {$\nabla f \in \mathbb{R}^{2D}$};
        \end{tikzpicture}
        \caption{Gradient vector extraction}
        \label{fig:grad_extraction}
    \end{subfigure}
    \hfill
    % --- Figure (b): Hessian Block Partitioning ---
    \begin{subfigure}[b]{0.48\textwidth}
        \centering
        \begin{tikzpicture}[scale=0.85]        
            % --- The Matrix Blocks ---
            % Top Left (Blue)
            \draw[thick, fill=blue!10] (0,2) rectangle (2,4);
            \node at (1, 3) {$H^{(1,1)}$};
            
            % Top Right (Orange)
            \draw[thick, fill=orange!10] (2,2) rectangle (4,4);
            \node at (3, 3) {$H_{(1,2)}$};
            
            % Bottom Left (Orange)
            \draw[thick, fill=orange!10] (0,0) rectangle (2,2);
            \node at (1, 1) {$H^{(2,1)}$};
            
            % Bottom Right (Blue)
            \draw[thick, fill=blue!10] (2,0) rectangle (4,2);
            \node at (3, 1) {$H^{(2,2)}$};
            
            % --- Indices ---
            % Left Indices (shifted left for consistency)
            \node[anchor=east] at (-0.3, 3) {\scriptsize $1:D$};
            \node[anchor=east] at (-0.3, 1) {\scriptsize $D+1:2D$};
            
            % Top Indices
            \node[anchor=south] at (1, 4) {\scriptsize $1:D$};
            \node[anchor=south] at (3, 4) {\scriptsize $D+1:2D$};
            
        \end{tikzpicture}
        \caption{Hessian block partitioning}
        \label{fig:Hessian_partition}
    \end{subfigure}
    
    \caption{Visual representation of the extraction of the sub-gradient $g$ and Hessian blocks.}
    \label{fig:grad_hess_blocks}
\end{figure}

\begin{convention}
    In our work, we predominantly evaluate derivatives at $z \in \rdr$ with $\zo = \zt \in \mb{R}^{D}$. To emphasize the coincidence, we use the notation $\zz \in \mb{R}^{D}$.
\end{convention}

\begin{lemma} [Properties of Gradient and Hessian of Skew-Symmetric Functions at $(z,z)$] \label{lem:sign_of_gradient_and_Hessian_blocks} 
Let $f: \mb{R}^{2D} \to \mb{R}$ be a skew-symmetric function, and let the required derivatives exist.
  When $\zo  =\zt = \zz$, the gradient and Hessian blocks of the skew-symmetric function $f$ satisfy the following relationships:
  \begin{align*}
    f(\zz,\zz) &= 0 \in \mb{R} \\
    \grad^{(1)} f(\zz,\zz) &= - \grad^{(2)} f(\zz,\zz) \in \mb{R}^D \\
    H^{(1,1)}(\zz,\zz) &= H^{(1,1)}(\zz,\zz)^\tr, \quad  H^{(2,2)}(\zz,\zz) = H^{(2,2)}(\zz,\zz)^\tr \in \mb{R}^{D \times D} \\
    H^{(1,1)}(\zz,\zz) &= - H^{(2,2)}(\zz,\zz) \in \mb{R}^{D \times D} \\
    H^{(1,2)}(\zz,\zz) &= H^{(2,1)}(\zz,\zz)^\tr = - H^{(1,2)}(\zz,\zz)^\tr = - H^{(2,1)}(\zz,\zz) \in \mb{R}^{D \times D}    
  \end{align*}
\end{lemma}
The proof of this result is given in \cite{cebra2022similaritysuppressescyclicitysimilar}.
With this equality, all results introduced hereafter are presented with respect to the extracted gradient $g$, diagonal Hessian blocks $\hoo$, and off-diagonal Hessian blocks $\hot$.
Note that the diagonal block is a symmetric matrix and the off-diagonal block is a skew-symmetric (anti-symmetric) matrix.

% \textit{I would add a remark here that restates Lemma 18 in plain English (e.g., the diagonal blocks of the Hessian are symmetric, the off-diagonal blocks are skew-symmetric, and they are otherwise equal).}

Let $\tilde{f}$ be the skew-symmetric second-order Taylor approximation of $f$ at the point $z = (\zz, \zz) \in \rdr$. The expression can be derived as follows from \ref{lem:sign_of_gradient_and_Hessian_blocks}. The expression is derived in \cite{cebra2022similaritysuppressescyclicitysimilar} with greater detail.
\begin{equation} \label{eq:Taylor_approximation} 
\begin{aligned}
    \tilde{f}(\xo, \xt)  
    &\simeq f(\zz, \zz) + (\grad f(\zz, \zz))^\tr \begin{bmatrix} \xo - \zz \\ \xt - \zz \end{bmatrix} + \frac{1}{2} \begin{bmatrix}  \xo - \zz \\ \xt - \zz \end{bmatrix}^\tr \grad^2 f(\zz, \zz) \begin{bmatrix} \xo - \zz \\ \xt - \zz \end{bmatrix}  \\
    &\simeq g^\tr \big( (\xo - \zz) - (\xt - \zz)\big) + \frac{1}{2} \bigg((\xo - \zz)^\tr H^{(1,1)} (\xo - \zz) -  (\xt - \zz)^\tr H^{(1,1)} (\xt - \zz) \bigg)  \\ 
    & \hdots +(\xo - \zz)^\tr H^{(1,2)} (\xt - \zz) 
\end{aligned}
\end{equation}

%&= \sum_{i=1}^D(\der^{(1)}_i f)\cdot (\xo_i - \xt_i) + \frac{1}{2} \sum_{i,j = 1}^D  \Bigg( \big(\der^{(1,1)}_{(i,j)}f \big) \cdot \bigg( ((\xo_i - z)(\xo_j - z) - (\xt_i - z)(\xt_j - z)) \bigg)\Bigg) + \nonumber\\ 
    %&\quad \ldots + \sum_{i,j = 1}^D \big(\der^{(1,2)}_{(i,j)} f \big) \cdot (\xo_i - z)(\xt_j - z) \nonumber

    % &= f(z, z) + (\go f)^\tr (\xo - z) + (\gt f)^\tr (\xt - z) + \frac{1}{2} (\xo - z)^\tr H^{(1,1)} (\xo - z) + \frac{1}{2} (\xt - z)^\tr H^{(2,2)} (\xt - z) \\
    % &+ \frac{1}{2} (\xo -z)^\tr H^{(1,2)} (\xt - z) + \frac{1}{2} (\xt - z)^\tr H^{(2,1)} (\xo - z) + \mc{O}((\xo-z, \xt-z)^3)\\
    % &= (\go f)^\tr (\xo - z)  - (\go f)^\tr (\xt - z) +  \frac{1}{2} (\xo - z)^\tr H^{(1,1)} (\xo - z) - \frac{1}{2} (\xt - z)^\tr H^{(1,1)} (\xt - z) \\
    % &+ (\xo - z)^\tr H^{(1,2)} (\xt - z) + \mc{O}((\xo-z, \xt-z)^3) \\

\begin{proposition} \label{prop:Taylor_approximation_is_gp}
    Let $f \sim GF(\mu_f, k_f)$ be a skew-symmetric Gaussian field over $\rdr$ and $\tilde{f}$ be its second-order Taylor approximation at the point $z = (\zz, \zz) \in \rdr$ as defined in equation \ref{eq:Taylor_approximation}. Then, $\tilde{f}$ is also a Gaussian field.\\
    In particular, when $f$ is mean zero and $k_f$ is twice differentiable in $L^2$, the mean function of $\tilde{f}$ is also zero.
\end{proposition}
\begin{proof}
    Any linear combination of deterministic bases whose coefficients are Gaussian random variables satisfies the definition of a Gaussian field, due to Proposition \ref{prop:Gaussianity_ops}. 
    In the entrywise expression of the quadratic Taylor approximation above, each coefficient is a Gaussian random variable by Proposition \ref{prop:Gaussanity_under_differential_operator}, 
\end{proof}

% \subsection{Distributional Properties between Random Vector $g$ and Random Matrices $H^{(1,1)}, H^{(1,2)}$}
% \subsection{Distributional Characterization and Independence Properties of the Random Vector $g$ and Random Matrices $H^{(1,1)}$, $H^{(1,2)}$}
\subsection{Distribution of the Local Geometrical Structures}
\begin{theorem} [Characterization of the Gradient and Hessian Blocks] \label{thm:Covariance_structure}
Let $f \sim GF(0, k_f)$ be a block isotropic Gaussian field over $\rdr$ with a covariance function $k_f$ which is four times differentiable in $L^2$. Let $g, H^{(1,1)}, H^{(1,2)}$ be the gradient and Hessian blocks of $f$ evaluated at the point $z = (\zz, \zz) \in \rdr$ as defined in \ref{notation_gradient_Hessian}. 
    \begin{align*}
        &\Cov(g_i, g_j) = -4h'(0) \cdot \dl_{ij} \\
        &\Cov(H^{(1,1)}_{i,j}, H^{(1,1)}_{k,l}) = 8h''(0) \cdot (\dl_{ik}\dl_{jl}+\dl_{il}\dl_{jk}) \\
        &\Cov(H^{(1,2)}_{i,j}, H^{(1,2)}_{k,l}) = 8h''(0) \cdot (\dl_{ik}\dl_{jl}-\dl_{il}\dl_{jk})
    \end{align*}
\end{theorem}
\begin{proof}
Due to Proposition \ref{prop:Covariance_of_partial_derivatives}, the covariance between derivatives equals the derivatives of the covariance function, thereby
\begin{align*}
    &\Cov(g_i, g_j) =\der^{(1)}_i \der^{(1)}_j k_f \big( (\zz, \zz) (\zz, \zz) \big) \quad\\
    &\Cov(H^{(1,1)}_{i,j}, H^{(1,1)}_{k,l}) = \der^{(1,1)}_{(i,j)} \der^{(1,1)}_{(k,l)} k_f \big( (\zz, \zz), (\zz, \zz) \big) \\
    &\Cov(H^{(1,2)}_{i,j}, H^{(1,2)}_{k,l}) = \der^{(1,2)}_{(i,j)} \der^{(1,2)}_{(k,l)}  k_f \big( (\zz, \zz), (\zz, \zz) \big) 
\end{align*}
Technical details are available in appendix \ref{appendix:Covariance_derivations}
\end{proof}

\begin{theorem} [Derivatives of Different Orders Are Uncorrelated]\label{thm:entrywise_uncorrelation}
    %The Covariance between any pair of values $f$, entries of $g$ and $H^{(1,1)}, H^{(1,2)}$ as follows.
    \begin{align*}
        &\Cov(g_i, f) = 0, \quad \Cov(\hoo_{i,j}, f) = 0 , \quad \Cov(\hot_{i,j}, f) = 0 \quad \all \text{ } i = 1, \cdots, D\\
        &\Cov(\hoo_{i,j}, g_k) = 0, \quad   \Cov(\hot_{i,j}, g_k ) = 0 \quad \all \text{ } i,j,k = 1, \cdots, D\\
        &\Cov(\hot_{i,j}, \hoo_{k,l}) = 0 \quad \all \text{ } i,j,k,l = 1, \cdots, D
    \end{align*}
\end{theorem}
\begin{proof}
    Similarly, we use \ref{prop:Covariance_of_partial_derivatives} and derive similar quantities. Details are in appendix \ref{appendix:Covariance_derivations}.
\end{proof}

%In the next theorem, we consider the random vector $vec(\hoo), vec(\hot) \in \mb{R}^{D^2}$ obtained by vectorizing the random matrix $\hoo, \hot$ using standard $vec: \mb{R}^{D \times D} \to \mb{R}^{D^2}$ operation.

\begin{proposition}
    Consider the same setting as \ref{thm:entrywise_uncorrelation}. Then, the random vector
    \begin{align*}
        \bigg(g, vec(\hoo), vec(\hot) \bigg)^\tr \tx{ is jointly normal } \in L^2\bigg(\Omega; \mb{R}^{D + D^2 + D^2} \bigg)
    \end{align*}
    Consequently, any subvector of $ \bigg(g, vec(\hoo), vec(\hot) \bigg)^\tr$ is also jointly normal.
\end{proposition}
\begin{proof}
    The result follows directly from \ref{prop:Gaussianity_ops} since $\bigg(g, vec{\hoo}, vec(\hot) \bigg)^{\tr}$ are obtained by applying linear operators (partial derivatives and vectorization) to the Gaussian field $f$.
    A similar argument holds for any subvectors. 
\end{proof}

% \begin{corollary}
%     Consider the same setting as above \ref{thm:entrywise_uncorrelation}. Then, the random vector $g \in L^2(\Omega; \mb{R}^D), vec(\hoo), vec(\hot) \in L^2(\Omega; \mb{R}^{D \times D})$ are all pairwise uncorrelated.
% \end{corollary}
% \begin{proof}
%     The result follows directly from \ref{thm:entrywise_uncorrelation} since all the Covariance terms between each pair of $g, vec(\hoo), vec(\hot)$ are zero.
% \end{proof}

\begin{theorem} \label{thm:independence_structure}
    Under the same setting as above \ref{thm:entrywise_uncorrelation}, the random vector $g$ and random matrices $H^{(1,1)}, H^{(1,2)}$ are mutually independent. i.e.,
    \begin{align*}
        g \indp H^{(1,1)}, \quad g \indp H^{(1,2)}, \quad H^{(1,1)} \indp H^{(1,2)}
    \end{align*}
\end{theorem}
\begin{proof}
    Being jointly normal and uncorrelated implies independence. Then the result follows from Borel Isomorphism between $L^2(\Omega; \mb{R}^{D^2})$ and $L^2(\Omega; \mb{R}^{D \times D})$.
\end{proof}

\begin{theorem} \label{thm:goe_matrix}
    Under the same setting as above, $g \sim \mc{N}(0, I_{D})$, $H^{(1,1)} \sim GOE(D)$, $H^{(1,2)} \sim GSSE(D)$, up to scaling.
\end{theorem}
\begin{proof}
    For a mean-zero Gaussian field $f$, each entry of the gradient $g$ follows a Gaussian distribution, by Proposition \ref{prop:Gaussanity_under_differential_operator}. 
    Hence, the vector $g$ follows a mean-zero multivariate Gaussian distribution, and its covariance matrix is given by $-4h'(0) I_D$ by \ref{thm:Covariance_structure}.
    \bigskip
    Furthermore, entries of random matrices are independent and distributed as follows:
    \begin{align*}
    &[\hoo]_{(i,j)} = [\hoo]_{(j,i)}, \text{ } [\hot]_{(i,j)} = -[\hoo]_{(j,i)} \tx{ with probability one. } \tx{Furthermore, } \\
    &[\hoo]_{(i,j)} = \begin{cases}
        8h''(0) \cdot A_{(i,j)}  \tx { if } i >  j\\ 
        16h''(0) \cdot A_{(i,j)}  \tx { if } i =j\\ 
    \end{cases} \quad \quad 
    [\hot]_{(i,j)} = \begin{cases}
        8h''(0) \cdot B_{(i,j)} \tx { if } i > j\\ 
        0 \tx{ if }i =j\\ 
    \end{cases} \\
    &\text{where } A_{(i,j)}, B_{(i,j)} \overset{i.i.d}{\sim} \mc{N}(0,1)
    \end{align*}
    Hence, $\hoo, \hot$ are draws from  $GOE(D)$ and $GSSE(D)$ scaled by $\sqrt{2} \cdot 8h''(0)$. 
\end{proof}

Finally, we present the closed-form expression for the covariance function of the Taylor approximation $\tilde{f}$, i.e., $k_{\tilde{f}}: \rdr \times \rdr \to \mb{R}$ defined by
\begin{align*}
    k_{\tilde{f}} \big( (\xo, \xt), (\yo, \yt) \big) := \Cov\big( \tilde{f}(\xo, \xt), \tilde{f}(\yo, \yt) \big) \quad \all \text{ } (\xo, \xt), (\yo, \yt) \in \rdr
\end{align*}
\begin{proposition}[Covariance function of the Quadratic Approximation] \label{prop:Covariance_function_Taylor_approximation}
    Under the same setting as Theorem \ref{thm:entrywise_uncorrelation} and where $h$ is the scalar function associated with the block isotropic covariance function $k_f$ as in \ref{notation:kernel_function_h}, the covariance function of the Taylor approximation $\tilde{f}$ at the matched input point $z$ is expressed as follows.
    \begin{align*}
        k_{\tilde{f}} \big( x, y  \big) &= -4h'(0) \cdot \bigg(\big \langle \Delta x, \Delta y \big \rangle_2 - \big \langle \Delta x, (\Delta y)^{flip} \big \rangle_2  \bigg) + 4h''(0) \cdot \bigg(\big \langle \Delta x, \Delta y \big \rangle_2^2 - \big \langle \Delta x, (\Delta y)^{flip} \big \rangle_2^2 \bigg)  \\
        &\tx{where } \Delta x = x - z, \Delta y = y - z
    \end{align*}
        % -4h'(0) \cdot \bigg(\big \langle x, y \big \rangle_2 - \big \langle x, y^{flip} \big \rangle_2  \bigg) + 4h''(0) \cdot \bigg(\big \langle x, y \big \rangle_2^2 - \big \langle x, y^{flip} \big \rangle_2^2 \bigg) 
\end{proposition}

\begin{proof}
    Recall that the quadratic Taylor approximation at $z= (\zz, \zz)$ is expressed as 
    \begin{align*}
        &\tilde{f}(\xo, \xt) = g^\tr \big( \vo - \vt \big) + \frac{1}{2} \bigg((\vo)^\tr \hoo \vo -  (\vt)^\tr \hoo \vt \bigg) +(\vo)^\tr \hot \vt \\
        &\tilde{f}(\yo, \yt) = g^\tr \big( \too - \ttt \big) + \frac{1}{2} \bigg((\too)^\tr \hoo \vo -  (\ttt)^\tr \hoo \ttt \bigg) +(\too)^\tr \hot \ttt \\
        &\tx{where }  \vo := \xo - z^*, \quad \vt := \xt - z^* , \quad \too := \yo - z^*, \quad \ttt := \yt - z^*
    \end{align*}
We expand $k_{\tilde{f}}(x,y) = \Cov\big( \tilde{f}(\xo, \xt), \tilde{f}(\yo, \yt) \big)$ term by term.
First, by Theorem \ref{thm:independence_structure}, note that all cross-covariance terms involving $g$ and either $\hoo$ or $\hot$ vanish.
% To see this, consider one of them as an example:
% \begin{align*}
%     &Cov(g^\tr \big( \vo - \vt \big), (\too)^\tr \hoo \ttt) \\
%     &= Cov \bigg( \sum_{i=1}^D g_i (\vo_i - \vt_i), \sum_{j,k=1}^D \too_j \ttt_k \hoo_{j,k} \bigg) \\
%     &= \sum_{i=1}^D \sum_{j,k=1}^D (\vo_i - \vt_i) \too_j \ttt_k Cov(g_i, \hoo_{j,k}) = 0 \quad \tx{ from \ref{thm:entrywise_uncorrelation} }
% \end{align*}

Then it remains to compute the following three terms using Theorem \ref{thm:Covariance_structure}. For further details, see appendix \ref{appendix:Covariances_interacting_terms}.
\begin{align*}
    &\text{1. }\Cov\big( g^\tr \big( \vo - \vt \big), g^\tr \big( \too - \ttt \big) \big) \\
    &=-4h'(0) \bigg( \big \langle \vo, \too \big \rangle_2 + \big \langle \vt, \ttt \big \rangle_2 - \big \langle \vo, \ttt \big \rangle_2 - \big \langle \vt, \too \big \rangle_2 \bigg) \\
    \bigskip 
    &\text{2. }\Cov\bigg( \frac{1}{2} \big((\vo)^\tr \hoo \vo -  (\vt)^\tr \hoo \vt \big), \frac{1}{2} \big((\too)^\tr \hoo \too -  (\ttt)^\tr \hoo \ttt \big) \bigg)  \\
    &= 4h''(0) \cdot \bigg( ||(\vo)^\tr \too||_2^2 + ||(\vt)^\tr \ttt||_2^2 - ||(\vo)^\tr \ttt||_2^2 - ||(\vt)^\tr \too||_2^2 \bigg)\\
    &\text{3. }\Cov\big( (\vo)^\tr \hot \vt, (\too)^\tr \hot \ttt \big)  \\
    &= 4h''(0) \cdot  \bigg( 2 \big \langle \vo, \too \big \rangle_2 \cdot \big \langle \vt, \ttt \big \rangle_2 - 2 \big \langle \vo, \ttt \big \rangle_2 \cdot \big \langle \vt, \too \big \rangle_2 \bigg) 
\end{align*}
Let $A = \big \langle \vo, \too \big \rangle_2, \quad B = \big \langle \vo, \ttt \big \rangle_2, \quad C = \big \langle \vt, \too \big \rangle_2, \quad D = \big \langle \vt, \ttt \big \rangle_2$. Then, collecting terms:
\begin{align*}
    &k_{\tilde{f}} \big( (\xo, \xt), (\yo, \yt) \big) = -4h'(0) \cdot (A + D - B - C) + 4h''(0) \cdot (A^2 + D^2 - B^2 - C^2 + 2AD - 2BC) \\
    &= -4h'(0) \cdot \big((A + D) - (B + C) \big) + 4h''(0) \cdot \big((A + D)^2 - (B + C)^2 \big) \\
    &= -4h'(0) \cdot \bigg(\big \langle \Delta x, \Delta y \big \rangle_2 - \big \langle \Delta x, (\Delta y)^{flip} \big \rangle_2  \bigg) + 4h''(0) \cdot \bigg(\big \langle \Delta x, \Delta y \big \rangle_2^2 - \big \langle \Delta x, (\Delta y)^{flip} \big \rangle_2^2 \bigg) 
    % &= -4h'(0) \cdot \bigg(\big \langle (x-z), (y-z) \big \rangle_2 - \big \langle (x-z), (y-z)^{flip} \big \rangle_2  \bigg) + 4h''(0) \cdot \bigg(\big \langle (x-z), (y-z) \big \rangle_2^2 - \big \langle (x-z), (y-z)^{flip} \big \rangle_2^2 \bigg) 
\end{align*}
\end{proof}

% But due to stationarity of $k_f: \rdrtr$, $k_f(x-z, y-z) = k_f(x,y)$. Hence,
% \begin{align*}
%     k_{\tilde{f}} \big( x, y  \big) &= -4h'(0) \cdot \bigg(\big \langle x, y \big \rangle_2 - \big \langle x, y^{flip} \big \rangle_2  \bigg) + 4h''(0) \cdot \bigg(\big \langle x, y \big \rangle_2^2 - \big \langle x, y^{flip} \big \rangle_2^2 \bigg) 
% \end{align*}

\begin{corollary}[Invariance of $k_{\tilde{f}}$ under Block Orthogonal transformation $Q^b: \rdr \to \rdr$] \label{proposition:invariance_of_tilde}
    Then for any block orthogonal transformation $Q^b: \rdr \to \rdr$,
    \begin{align*}
        \{\tilde{f}(Q^b x) \}_{x \in \rdr} \overset{D}{=} \{\tilde{f}(x)\}_{x \in \rdr}.
    \end{align*}
\end{corollary}

\begin{proof}  
    Since $\tilde{f}$ is a mean-zero Gaussian field, it suffices to show that the covariance function $k_{\tilde{f}}$ is invariant under the block orthogonal transformation $Q^b: \rdr \to \rdr$.
    For any $Q^b: \rdr \to \rdr$,
    \begin{align*}
        \big \langle Q^b x, Q^b y \big \rangle^2_2 &= \big \langle x, y \big \rangle^2_2 \\
        \big \langle Q^b x, Q^b y^{flip} \big \rangle^2_2 &= \big \langle x, y \big \rangle^2_2 
    \end{align*}
    The results follow directly from the definition of block orthogonal transformation and the properties of the inner product.
\end{proof}

\begin{remark}
    The orthogonal invariance result in Proposition~\ref{proposition:invariance_of_tilde} can be understood as a consequence of the orthogonal invariance of the random vector $g$ and the random matrices $H^{(1,1)}$ and $H^{(1,2)}$. That invariance follows from the same invariance of $f$ since $g$ and $H$ are derived from $f$.
\end{remark}

\section{Error induced by Quadratic Approximation} \label{sec: error}
\subsection{Error Fields of Block Isotropic Gaussian Fields}
For deterministic functions, the discrepancy between the function and its Taylor approximation increases with the distance from the expansion point. 
Then, for random fields, we expect the discrepancy to increase "stochastically" with distance.
In general, characterizing the distribution of the error induced by the Taylor approximation of random fields is challenging.
However, for Block Isotropic Gaussian fields, due to their intrinsic structure, the error is also a Gaussian field, and the pointwise error is a mean-zero Gaussian random variable. 
The variance of the error process at each input completely determines the distribution of the pointwise error and is available in closed form.
Furthermore, the maximal error over a neighborhood of the location $z^*$ used as the center of the approximation can be bounded from above with high probability. We extend this bound to randomly distributed inputs drawn from Gaussian distributions centered at $z^*$.

\begin{convention}
As noted in \ref{convention:mean_zero_base_field}, we assume the mean function of skew-symmetric Gaussian fields to be zero unless otherwise stated.
\end{convention}

\begin{definition}[Error Process]
    Let $f \sim GF(0, k_f)$ be a skew-symmetric Gaussian field over $\rdr$ with $k_f \in C^4((\rdr) \times (\rdr))$ differentiability and $\tilde{f}$ be its second-order Taylor approximation at the point $z = (\zz, \zz) \in \rdr$.
    Then, we define the error process $\mc{E}: \rdr \to \mb{R}$ as
    \begin{align*}
        % \mc{E}(\xo, \xt) := f(\xo, \xt) - \tilde{f}(\xo, \xt) \quad \all \text{ } (\xo, \xt) \in \rdr
        \mc{E}(\xo, \xt) := \tilde{f}(\xo, \xt) - f(\xo, \xt) \quad \all \text{ } (\xo, \xt) \in \rdr
    \end{align*}
\end{definition}

\begin{proposition}
    Under the same setting as above, the error process $\mc{E}$ is also a Gaussian field over $\rdr$ with mean zero.
\end{proposition}
\begin{proof}
    Since $\tilde{f}$ is a Gaussian field as shown in \ref{prop:Taylor_approximation_is_gp} and $f$ is also a Gaussian field, their difference $\mc{E}$ is also a Gaussian field by \ref{prop:Gaussianity_ops}.
    Since both $f$ and $\tilde{f}$ are mean zero, $\mc{E}$ is also mean zero.
\end{proof}

Since $\mc{E}$ is a mean-zero Gaussian field, its distribution is fully determined by its covariance function.
Hereafter, we assume $f$ to be a block isotropic Gaussian field with its covariance function $k_f$ being four times differentiable in $L^2$.
\begin{theorem} \label{thm:error_orthogonal_invariance}
    Under the same setting as above, the error process $\mc{E}$ is invariant under block orthogonal transformation $Q^b: \rdr \to \rdr$.
    \begin{align*}
        \{\mc{E}(Q^b x\}_{x \in \rdr} \overset{d}{=} \{\mc{E}(x)\}_{x \in \rdr}.
    \end{align*}
\end{theorem}
\begin{proof}
    Under the assumptions, both $f$ and $\tilde{f}$ are invariant under block orthogonal transformation $Q^b: \rdr \to \rdr$ as shown in \ref{proposition:invariance_of_tilde} and \ref{proposition:orthogoanal_invariance_of_f}.
    Then their difference must share the same invariances.
\end{proof}

We now present the closed-form expression for the covariance function of the error process $\mc{E}$, i.e., $k_{\mc{E}}: \rdr \to \mb{R}^D$ defined by
\begin{align*}
    k_{\mc{E}} \big( (\xo, \xt), (\yo, \yt) \big) := \Cov \big( \mc{E}(\xo, \xt), \mc{E}(\yo, \yt) \big) \quad \all \text{ } (\xo, \xt), (\yo, \yt) \in \rdr.
\end{align*}
The covariance function of the error process can be expressed in terms of the covariance function of the base field. 
Let $k_u: \rdr \to \mathbb{R}$ denote the covariance function of the base field, which satisfies $f(x_1, x_2) = u(x_1, x_2) - u(x_2, x_1)$. For Block Isotropic Gaussian fields (cf. Definition \ref{definition:block_isotropic_gaussian_field}), the covariance functions of the base fields are isotropic. They can thus be expressed as $h: \mathbb{R}_{\geq0} \to \mathbb{R}$ as we introduced in Notation \ref{notation:kernel_function_h}.
%where $h$ depends only on the squared norm of the difference between its arguments.
 
\begin{proposition} \label{thm:Covariance_function_error_process}
    Here, $h: \mb{R}_{\geq 0} \to \mb{R}$ is the scalar function associated with the covariance function of the base field. Then, the covariance function of the error process $\mc{E}$ is:
    \begin{align*}
        k_{\mc{E}} \big( x, y \big) &= 2 \bigg( h\big( \|x - y\|_2^2  - h\big( \|x - y^{flip}\|_2^2 \big) \bigg) + 4 \bigg(h'\big(\|\Delta x\|_2^2\big) + h'\big(\|\Delta y\|_2^2\big) - h'\big( 0 \big)\bigg) \bigg(\langle \Delta x, \Delta y \big \rangle_2 - \big \langle \Delta x, (\Delta y)^{flip} \big \rangle_2\bigg) \\
    &\quad  -4 \bigg(h''\big(\|\Delta x\|_2^2\big) + h''\big(\|\Delta y\|_2^2\big) - h''\big( 0 \big)\bigg) \bigg(\langle \Delta x, \Delta y \big \rangle_2^2 - \big \langle \Delta x, (\Delta y)^{flip} \big \rangle_2^2\bigg)
    \end{align*}
\end{proposition}
        %&= Cov \big( f(\xo, \xt), f(\yo, \yt) \big) - Cov\big( f(\xo, \xt), \tilde{f}(\yo, \yt) \big) \\
        %&\quad\quad - Cov\big( \tilde{f}(\xo, \xt), f(\yo, \yt) \big) + Cov\big( \tilde{f}(\xo, \xt), \tilde{f}(\yo, \yt) \big)\\
\begin{proof}
    \begin{align*}
        &k_{\mc{E}} \big( (\xo, \xt), (\yo, \yt) \big) \\
        &= \Cov\big(  \tilde{f}(\xo, \xt) - f(\xo, \xt), \tilde{f}(\yo, \yt)  - f(\yo, \yt)\big) \\
        &=\Cov\big( \tilde{f}(\xo, \xt), \tilde{f}(\yo, \yt) \big)  - \Cov \big( \tilde{f}(\xo, \xt), \tilde{f}(\yo, \yt) \big) - \Cov\big( f(\xo, \xt), \tilde{f}(\yo, \yt) \big) \\
        &\quad \ldots + \Cov \big( f(\xo, \xt), f(\yo, \yt) \big)
    \end{align*}
The last term is the covariance function of the Gaussian field $f$, and the first terms have been derived in Proposition \ref{prop:Covariance_function_Taylor_approximation}. Hence, we need to obtain expressions for $\Cov \big( f(\xo, \xt), \tilde{f}(\yo, \yt) \big)$. Then $\Cov \big( (\xo, \xt), (\yo, \yt) \big)$ can be obtained similarly.
Again, we denote $\vo := \xo - \zz, \vt := \xt - \zz , \too := \yo - \zz, \ttt := \yt - \zz$ for brevity in the following derivation.
\begin{align*}
    &\Cov\big( f(\xo, \xt), \tilde{f}(\yo, \yt) \big) \\
    &= \Cov\bigg( f(\xo, \xt), g^\tr \big( \too - \ttt \big) + \frac{1}{2} \bigg((\too)^\tr \hoo \too -  (\ttt)^\tr \hoo \ttt \bigg) +(\too)^\tr \hot \ttt \bigg) \\
    &= \Cov\big( f(\xo, \xt), g^\tr \big( \too - \ttt \big) \big) + \Cov\bigg( f(\xo, \xt),  \frac{1}{2} \bigg((\too)^\tr \hoo \too -  (\ttt)^\tr \hoo \ttt \bigg) \bigg) \\
    &\quad \ldots + \Cov\big( f(\xo, \xt), (\too)^\tr \hot \ttt \big) 
\end{align*}
 As in the previous proof, we handle each term separately and obtain the following expressions. Detailed derivations are provided in Appendix \ref{appendix:Covariance_of_error_process}.
 \begin{align*}
\Cov\big( f(\xo, \xt), \tilde{f}(\yo, \yt) \big) = -4h'\big( \|\Delta x\|_2^2 \big) \cdot \bigg(\big \langle \Delta x, \Delta y \big \rangle_2 - &\big \langle \Delta x, (\Delta y)^{flip} \big \rangle_2  \bigg) \\
&\ldots + 4h''\big( \|\Delta x\|_2^2 \big) \cdot \bigg(\big \langle \Delta x, \Delta y \big \rangle_2^2 - \big \langle \Delta x, (\Delta y)^{flip} \big \rangle_2^2 \bigg) \\
\Cov(\tilde{f}(\xo,\xt), f(\yo,\yt)) = -4h'\big( \|\Delta y\|_2^2 \big) \cdot \bigg(\big \langle \Delta x, \Delta y \big \rangle_2 - &\big \langle \Delta x, (\Delta y)^{flip} \big \rangle_2  \bigg) \\
&\ldots + 4h''\big( \|\Delta y\|_2^2 \big) \cdot \bigg(\big \langle \Delta x, \Delta y \big \rangle_2^2 - \big \langle \Delta x, (\Delta y)^{flip} \big \rangle_2^2 \bigg)
\end{align*}
Combining terms produces the desired result.
\end{proof}

\subsection{Pointwise Error Bound}
As noted earlier, when $f$ is a Gaussian field, $f_x$ is a real-valued Gaussian random variable. Similarly, $\mc{E}(x)$ is a Gaussian random variable since $\mc{E}$ is a Gaussian field, i.e.,
\begin{align*}
    \mc{E}(\xo, \xt) \sim \mc{N} \big(0, \sigma^2_{\mc{E}}(\xo, \xt) \big) \quad \all \text{ } (\xo, \xt) \in \rdr
\end{align*}
Since the error at any input is a mean-zero Gaussian random variable, its distribution is entirely determined by its standard deviation $\sigma_{\mc{E}}(\xo, \xt)$.

\begin{proposition} [Pointwise Variance of the Error Process] \label{prop:pointwise_variance_error_process}
    Under the same setting as \ref{thm:Covariance_structure}, for fixed $(\xo, \xt) \in \rdr$, the pointwise error $\mc{E}(\xo, \xt) \in L^2(\Omega)$ is a mean-zero Gaussian random variable with its variance given as
    \begin{align*}
        \sigma^2_{\mc{E}}(\xo, \xt))  &= 2 \bigg( h\big( 0 \big) - h\big( 2\|\xo - \xt\|_2^2 \big) \bigg) -4 \bigg( h'\big( 0 \big) - 2h'\big( \|\xo|_2^2 + \|\xt\|_2^2 \big) \bigg) \cdot \|\xo - \xt\|_2^2 \\
        &+4 \bigg(h''\big( 0 \big) - 2h''\big( \|\xo\|_2^2 + \|\xt\|_2^2 \big) \bigg)  \cdot \|\xo - \xt\|_2^2 \cdot \|\xo + \xt\|_2^2 \\
    \end{align*} 
\end{proposition}

\begin{proof}
    The result follows directly from \ref{thm:Covariance_function_error_process} by substituting $(\yo, \yt) = (\xo, \xt)$, then each term simplifies as follows:
\begin{enumerate}
    \item $\|x - y\|_2^2 = 0$ and $\|x - y^{flip}\|_2^2 =  \|x - x^{flip}\|_2^2 = 2\|\xo - \xt\|_2^2$
    \item $\|y\|_2^2 =  \|x\|_2^2  = \|\xo\|_2^2 + \|\xt\|_2^2$
    \item $\langle x, y \rangle_2 - \langle x, y^{flip}  \rangle_2 = \langle x, x - x^{flip} \rangle_2 = |\xo\|_2^2 - \langle \xo, \xt \rangle_2 - \langle \xt, \xo \rangle_2 + \|\xt\|_2^2 = \|\xo - \xt \|_2^2$
    \item $\langle x, y \rangle_2^2 -  \langle x, y^{flip} \rangle_2^2 = \langle x, x \rangle_2^2 -  \langle x, x^{flip} \rangle_2^2 = \langle x, x - x^{flip}\rangle_2 \cdot \langle x, x + x^{flip} \rangle_2 = \|\xo - \xt \|_2^2 \cdot \|\xo + \xt \|_2^2$
\end{enumerate}
\end{proof}

% Then:
% \begin{align*}
%     k_{\mc{E}} \big( x, y \big)
%     &= 2 \bigg( h\big( 0 \big) - h\big( 2\|\xo - \xt\|_2^2 \big) \bigg) -4 \bigg( h'\big( 0 \big) - 2h'\big(\|x\|_2^2\big)\bigg) \bigg(\langle x, x - x^{flip} \rangle_2\bigg) \\
%     &\quad +4 \bigg(h''\big( 0 \big) - 2h''\big(\|x\|_2^2\big) \bigg)  \bigg(\langle x, x - x^{flip} \rangle_2 \cdot \langle x, x + x^{flip} \rangle_2\bigg) \\
%     &= 2 \bigg( h\big( 0 \big) - h\big( 2\|\xo - \xt\|_2^2 \big) \bigg) -4 \bigg( h'\big( 0 \big) - 2h'\big(\|x\|_2^2\big)\bigg) \bigg(\langle x, x - x^{flip} \rangle_2\bigg) \\
%     &\quad +4 \bigg(h''\big( 0 \big) - 2h''\big(\|x\|_2^2\big) \bigg)  \bigg(\langle x, x - x^{flip} \rangle_2 \cdot \langle x, x + x^{flip} \rangle_2\bigg)
% \end{align*}

\begin{theorem} [Equivalent Expressions defined over the plane] \label{thm:polar_expression}
    The pointwise variance in the error process, $\sigma^2_{\mc{E}}$, can be expressed in terms of the magnitudes $r_1, r_2$ of, and the relative angle $\theta$ between, the inputs.
    Specifically, 
    \begin{align*}
        \sigma^2_{\mc{E}}(r_1, r_2, \theta)
        &= 2 \bigg( h\big( 0 \big) - h\big( 2(r_1^2 + r_2^2 - 2r_1 r_2 \cos \theta ) \big) \bigg) -4 \cdot \bigg( h'\big( 0 \big) - 2h'\big( r_1^2 + r_2^2 \big) \bigg) \cdot \bigg(r_1^2 + r_2^2 - 2r_1 r_2 \cos \theta \bigg)  \\
        &+ 4 \cdot \bigg(h''\big( 0 \big) - 2h''\big( r_1^2 + r_2^2 \big) \bigg) \cdot \bigg((r_1^2 +r_2^2)^2 - 4 r_1^2r_2^2 cos^2\theta \bigg)
    \end{align*}
\end{theorem}

\begin{proof} 
    The pointwise error variance $\sigma^2_{\mc{E}}(\xo, \xt))$ is determined by $\xo, \xt$ solely through their norms and the angle between them. 
    Therefore, we can adopt the coordinate system such that $\xo$ is aligned with the first axis and $\xt$ lies in the plane spanned by the first two axes without loss of generality. 
    Then,
    \begin{align*}
        \xo &= (r_1, 0, \dots, 0)^\top \\
        \xt &= (r_2 \cos \theta, r_2 \sin \theta, 0, \dots, 0)^\top 
        \tx{where } r_1 = \|\xo\|_2, r_2 = \|\xt\|_2, \tx{and } \theta \tx{ is the angle between them.}
    \end{align*}
    The corresponding squared norm of the vector, its sum, and its difference are:
    \begin{align*}
        &\| x \|_2^2 = \| \xo \|_2^2 + \| \xt \|_2^2 = r_1^2 + r_2^2 \nonumber \\
        &\|\xo - \xt\|_2^2 = r_1^2 + r_2^2 - 2r_1 r_2 \cos \theta \nonumber \\
        &\|\xo + \xt\|_2^2 = r_1^2 + r_2^2 + 2r_1 r_2 \cos \theta \nonumber
    \end{align*}
    The result can be obtained by substituting this in the pointwise variance provided in Proposition \ref{prop:pointwise_variance_error_process}.
\end{proof}

\begin{definition} [Pointwise Error Bound] \label{def:pointwise_error_bound}
    For the confidenc level $p \in (0,1)$, the pointwise error bound is defined as the $p$-quantile:
    \begin{align*} 
    b^{*}_p\bb{\xo, \xt} := inf \cd{b \in \mb{R}: \mb{P} \dc{\mathcal{E} \bb{\xo, \xt} \leq b} > p}
    \end{align*}    
\end{definition}

\begin{proposition} [Decomposition of the Pointwise Error Bound] \label{proposition:decomposition_of_pointwise_error_bound}
\begin{align*}
b^{*}_p\bb{\xo, \xt} &= \sigma_{\mc{E}}(\xo, \xt) \cdot Q(p) \tx{ where } (\xo,\xt) \in \rdr \\
&Q: \tx{Quantile Function for Standard Normal Variable}
\end{align*}
\end{proposition}
\begin{proof}
It follows immediately from the definition of a standard normal variable.
\end{proof}

%\newpage
Local expansions are accurate only within specific neighborhoods.
Whenever local expansions are used, it is essential to understand how the worst-case error varies with the size of the neighborhood. We will study this in two ways: first, by considering the maximal error over elliptical neighborhoods centered at $z$, and second, by examining the uniform error bound for randomly chosen inputs drawn from a Gaussian distribution centered at $z$.
\subsection{Maximal Error over Ellipsoidal Region} \label{section: maximal_error_over_ellipsoid}
 % determined by the positive definite matrix $\Sigma\in \mb{R}^{D \times D}$.
 % \tx{ where } A := \big\{ x \in \rdr : x^{\tr} \Sigma^{-1} x \leq 1\big\} 
\begin{definition} [Maximal Error over the Region]
    We define $b^*_p(A)$ as the supremum of error over the region $A \subseteq \mb{R}^{D}$.
    \begin{align*}
        &i.e., \text{ } b^*_p(A) := \sup_{\xo, \xt \in A} b^*_p(\xo, \xt)
    \end{align*}
\end{definition}

Due to Proposition \ref{proposition:decomposition_of_pointwise_error_bound}, the following equality holds.
\begin{align*}
    b^*_p(A) = Q(p) \cdot \sup_{\xo, \xt \in A} \sigma_{\mc{E}}(\xo, \xt) 
\end{align*}

Solving this optimization problem is challenging because the pointwise standard deviation of the error process depends on the function $h$, which determines the smoothness of the original function $f$, but has been left unspecified. So, we provide an upper bound instead of the exact maximum. 

\begin{remark}
The supremum of the pointwise standard deviation, $\sigma_{\mc{E}}$, over a set is monotone increasing with respect to the set inclusion,
\begin{align} \label{observation:set_inclusion_monotonicity}
    A \subseteq B \implies \sup_{\xo, \xt \in A}  \sigma_{\mc{E}}(\xo, \xt) \leq \sup_{\xo, \xt \in B}  \sigma_{\mc{E}}(\xo, \xt) 
\end{align}   
\end{remark}

The ellipsoidal region defined as $A:= \{x \in \mb{R}^D: x^{\tr}\Sigma^{-1}x \leq r^2 \}$ allows a simple closed-form expression for the upper bound. We use orthogonal invariance of $\mc{E}$ as introduced in Theorem \ref{thm:error_orthogonal_invariance} and monotonicity with respect to set inclusion. Our strategy is summarized as follows: First, rotate the coordinate system such that $\xo$ is aligned with the first axis and $\xt$ lies in the plane spanned by the first two axes. Second, since the pointwise error is monotone increasing with respect to set inclusion, the supremum is attained when both axes are aligned with the two largest principal axes of the ellipsoid, or equivalently, the two largest eigenvectors of $\Sigma$.
%This argument leads to the following proposition.
\begin{proposition} \label{proposition: hierarchy_of_maximization}
    \begin{align*}
        &\sup_{\xo, \xt \in A}  \sigma_{\mc{E}}(\xo, \xt) \leq \sup_{\xo, \xt \in A^{(\lambda_1, \lambda_2)}}  \sigma_{\mc{E}}(\xo, \xt) \leq \sup_{\xo, \xt \in A^{(\lambda_1, \lambda_1)}}  \sigma_{\mc{E}}(\xo, \xt) \\
        &\tx{ where } A^{(\lambda_1,\lambda_2) } := \bigg\{ (t_1, t_2)^\tr \in \mb{R}^2: \frac{t_1^2}{\lambda_1} + \frac{t_2^2}{\lambda_2} \leq r^2 \bigg\}, \quad A^{(\lambda_1,\lambda_1) } := \bigg\{ (t_1, t_2)^\tr \in \mb{R}^2: \frac{t_1^2}{\lambda_1} + \frac{t_2^2}{\lambda_1} \leq r^2 \bigg\}\\
        &\qquad \qquad \lambda_1 \geq \lambda_2 \tx{ are the two largest eigenvalues of } \Sigma.
    \end{align*}
\end{proposition}
\begin{proof}
    The first inequality immediately follows from the above argument.
    The second follows from set inclusion; $ A^{(\lambda_1, \lambda_2)} \subseteq  A^{(\lambda_1, \lambda_1)}$ since $\lambda_1 \geq \lambda_2$.
\end{proof}

Proposition \ref{proposition: hierarchy_of_maximization} reduces the original problem to a maximization problem over a circular set in $\mb{R}^2$. Then, we can use the polar expression of $\sigma_{\mc{E}}$ provided in Theorem \ref{thm:polar_expression}.
The problem is then reduced to the following form.
\begin{align*}
    \textbf{ Find: }&\sup_{0 \leq r_1, r_2 \leq R_D(p'), \theta \in [0, 2\pi)}  \sigma_{\mc{E}}(r_1, r_2, \theta), \quad \tx{ given } R_D(p') := \sqrt{\lambda_1 \cdot r^2}
\end{align*}

\subsubsection*{Construction of the Upper Bound for the Maximal Error over the Region}

From Theorem \ref{thm:polar_expression}, the variance of the pointwise error in polar form is:
\begin{align*}
    \sigma^2_{\mc{E}}(r_1, r_2, \theta) & = 2 \bigg( h\big( 0 \big) - h\big( 2(r_1^2 + r_2^2 - 2r_1 r_2 \cos \theta ) \big) \bigg) +4 \cdot \bigg( -h'\big( 0 \big) + 2h'\big( r_1^2 + r_2^2 \big) \bigg) \cdot \bigg(r_1^2 + r_2^2 - 2r_1 r_2 \cos \theta \bigg)  \\
    &+ 4 \cdot \bigg(h''\big( 0 \big) - 2h''\big( r_1^2 + r_2^2 \big) \bigg) \cdot \bigg((r_1^2 +r_2^2)^2 - 4 r_1^2r_2^2 cos^2\theta \bigg)
\end{align*}

An upper bound on $\sup_{\xo, \xt \in A^{(\lambda_1, \lambda_1)}}  \sigma_{\mc{E}}(\xo, \xt)$ can be constructed for the class of covariance functions $k_u: \rdr \times \rdr \to \mb{R}$ corresponding to a radial function $h: \mb{R}_{\geq0} \to \mb{R}$, as defined in Notation \ref{notation:kernel_function_h}, satisfies:
\begin{assumption}[Required Assumptions for Constructing Upper Bound] \label{assumption:restrction_on_h} 
\tx{} 
\begin{itemize}
    \item $h'(\cdot), h''(\cdot) \geq 0$ and is monotone decreasing, 
    \item $h'(\cdot) \leq 0$ and is monotone increasing, 
\end{itemize}
\end{assumption}

We introduce a stricter restriction. This concept and the relevant theorems are introduced in \cite{Wendland_2004}.
\begin{definition} [Completely Monotone Function]
A function $h: \mb{R}_{\geq0} \to \mb{R}$ is called completely monotone on $[0, \infty)$ if it satisfies $h \in C^{\infty}(0,\infty)$, $C[0,\infty)$ and
\begin{align*}
    (-1)^n h^{(n)}(t) \geq 0 , \quad \all \text{ } n \in \mb{N}, t \geq 0
\end{align*}    
\end{definition}

Notable classes of isotropic covariance functions, including the Squared Exponential, Matérn, and Rational Quadratic covariance functions, have corresponding functions $h: \mb{R}_{\geq 0} \to \mb{R}$ that are completely monotone.   Hence, they satisfy Assumption \ref{assumption:restrction_on_h}. This follows directly from the Bernstein–Hausdorff–Widder Theorem, since these covariance functions can be represented as mixtures of Gaussian kernels with respect to the Dirac delta distribution, the Inverse Gamma distribution, and the Gamma distribution, respectively. For further treatment for a kernel mixture model, see \cite{richland2024sharedendpointcorrelationshierarchyrandom}.

To simplify the analysis, we name each term of $\sigma_{\mc{E}}(r_1, r_2, \theta)$:
\begin{align*} 
    &\text{Term 1} = 2 \bigg( h\big( 0 \big) - h\big( 2(r_1^2 + r_2^2 - 2r_1 r_2 \cos \theta ) \big) \bigg) \\
    &\text{Term 2} = 4\cdot \text{(Term 2a)} \cdot \text{(Term 2b)} \tx{ where } \text{Term 2a} = \bigg(- h'\big( 0 \big) + 2h'\big( r_1^2 + r_2^2 \big) \bigg), \text{Term 2b} = \bigg(r_1^2 + r_2^2 - 2r_1 r_2 \cos \theta \bigg) \\
    &\text{Term 3} = 4\cdot \text{(Term 3a)} \cdot \text{(Term 3b)} \tx{ where } \text{Term 3a} = \bigg(h''\big( 0 \big) - 2h''\big( r_1^2 + r_2^2 \big) \bigg),
     \text{Term 3b} = \bigg((r_1^2 +r_2^2)^2 - 4 r_1^2r_2^2 cos^2\theta \bigg)
\end{align*}
Note that Term 2a and Term 3a may have different signs depending on the value of $r_1^2 + r_2^2$, while Term 2b and Term 3b are always non-negative.
We denote $R_{c^{1}}$ as the critical radius satisfying
\begin{align*}
    &- h'\big( 0 \big) + 2h'\big( R_{c^{1}}^2 ) = 0 
\end{align*}
Then we have
\begin{align*}
    &-h'\big( 0 \big) + 2h'\big( r_1^2 + r_2^2 \big) = | h'(0)| - |2h'\big( r_1^2 + r_2^2 \big)| = \begin{cases}
        \leq 0 & \tx{ if } r_1^2 + r_2^2 \leq R_{c^{1}}^2 \\
        > 0 & \tx{ if } r_1^2 + r_2 ^2 > R_{c^{1}}^2   
    \end{cases}
\end{align*}

Similarly, we consider $R_{c^{2}}$ as the critical radius defined by
\begin{align*}
    &h''\big( 0 \big) - 2h''\big( R_{c^{2}}^2 ) = 0 \
\end{align*}
This leads to
\begin{align*}
    &h''\big( 0 \big) - 2h''\big( r_1^2 + r_2^2 \big) = | h''(0)| - |2h''\big( r_1^2 + r_2^2 \big)| = \begin{cases}
        \leq 0 & \tx{ if } r_1^2 + r_2^2 \leq R_{c^{2}}^2 \\
        > 0 & \tx{ if } r_1^2 + r_2 ^2 > R_{c^{2}}^2   
\end{cases}
\end{align*}

Depending on the sign of each term, the pointwise standard deviation $\sigma_{\mc{E}}(r_1, r_2, \theta)$ exhibits different behavior. So, we construct the upper bound for each regime separately. 

We define three regimes:
\begin{itemize}
    \item Local Regime:= $\{(r_1, r_2) \in [0, R_D(p')) \times [0, R_D(p')): r_1^2 + r_2^2 \leq \min(R_{c^{1}}^2, R_{c^{2}}^2) \}$
    \item Transitional Regime:= $\{(r_1, r_2) \in [0, R_D(p')) \times [0, R_D(p')): \min(R_{c^{1}}^2, R_{c^{2}}^2) \leq r_1^2 + r_2^2  \leq \max(R_{c^{1}}^2, R_{c^{2}}^2)\}$
    \item Far Regime:= $\{(r_1, r_2) \in [0, R_D(p')) \times [0, R_D(p')): \max(R_{c^{1}}^2, R_{c^{2}}^2) \leq r_1^2 + r_2^2$
\end{itemize}

We provide in-depth analyses of the local and far regimes. Due to the monotonicity introduced in Observation \ref{observation:set_inclusion_monotonicity}, the supremum over transitional regimes $R_{c^{1}}^2 \leq r_1^2 + r_2^2 \leq R_{c^{2}}^2$ or $R_{c^{2}}^2 \leq r_1^2 + r_2^2 \leq R_{c^{1}}^2$ regimes can be bounded using the conclusion from the far regime.
Furthermore, one can conduct a similar analysis to derive an upper bound appropriate for the transitional regime.

Our approach can be summarized as:
\begin{align*}
    \max_{r_1,r_2} \bigg[ \max_{\theta \in [0, 2\pi)} \sigma_{\mc{E}} (r_1, r_2, \theta) \bigg]
\end{align*}
For a given $r_1, r_2$,  Term 1 is maximized when $\cos \theta = -1$, i.e., $\theta = \pi$ always. On the other hand, Term 2 achieves maximum at different $\theta$'s based on the regime, similarly for the Term 3. To address this issue, we give an upper bound by maximizing each term separately. This produces the alternate scheme:
\begin{align*}
    \max_{r_1,r_2}\bigg[ \max_{\theta \in [0, 2\pi)} \text{Term1} + \max_{\theta \in [0, 2\pi)} \text{Term2} + \max_{\theta \in [0, 2\pi)} \text{Term3} \bigg]
\end{align*}
We start with the local regime where $r_1^2 + r_2^2 \leq \min \{ R_{c^{1}}^2, R_{c^{2}}^2 \}$.

\begin{proposition} [Maximization over the Local Regime] \label{proposition: maximazation_local_regime}
    When $r_1^2 + r_2^2 \leq \min \{ R_{c^{1}}^2, R_{c^{2}}^2 \}$, 
    \begin{align*}
        \sup_{r_1, r_2 \in [0, R^2_D(p')], \theta \in [0, 2\pi)}  \sigma_{\mc{E}}(r_1, r_2, \theta)  \leq & 2 \bigg( h\big( 0 \big) - h\big( 8r^2 \big) \bigg) \tx{ where } r^2 = \min \bigg( R^2_{c^1}, R^2_{c^2}, R^2_D(p')\bigg) \\
    \end{align*}
\end{proposition}
\begin{proof}
    We assume that $r_1, r_2$ are fixed and only $\theta$ is variable. \\
    In this regime, both Term 2a and Term 3a are non-positive.
    Hence, Term 2 is maximized when Term 2b is minimized, which occurs at $\cos \theta = 1$, i.e., $\theta = 0$.
    Similarly, Term 3 is maximized when Term 3b is minimized, which also occurs at $\cos^2 \theta = 1$, i.e., $\theta = 0 \ or \  \pi$.

    For fixed $r_1, r_2$, Term 1 is maximized at $\theta = \pi$, Term 2 and Term 3 are maximized at $\theta = 0$.
    Hence, 
    \begin{align*}
     \sup_{r_1, r_2 \in [0, R^2_D(p')]} \sigma_{\mc{E}}(r_1, r_2, \theta)  
     &\leq \sup_{r_1, r_2 \in [0, R^2_D(p')]} \bigg[ \tx{Term 1 at } \theta = \pi + \tx{Term 2 at } \theta = 0 + \tx{Term 3 at } \theta = 0 \bigg] \\
    &= \sup_{r_1, r_2 \in [0, R^2_D(p')]} 2 \bigg( h\big( 0 \big) - h\big( 2(r_1 + r_2)^2 \big) \bigg) +4 \cdot \bigg( -h'\big( 0 \big) + 2h'\big( r_1^2 + r_2^2 \big) \bigg) \cdot \bigg(r_1 - r_2 \bigg)^2  \\
    &\quad \quad + 4 \cdot \bigg(h''\big( 0 \big) - 2h''\big( r_1^2 + r_2^2 \big) \bigg) \cdot \bigg(r_1^2 - r_2^2 \bigg)^2 \\
    \end{align*}
    Since Term 2a and Term 3a are non-positive, Term 2 and Term 3 are maximized when $r_1 = r_2$.
    On the other hand, Term 1 increases as $r_1^2 + r_2^2$ increases.
    Hence, the supremum is attained when $r_1^2 = r_2^2 = min \bigg( R^2_{c^1}, R^2_{c^2}, R^2_D(p')\bigg)$.
\end{proof}

\begin{corollary} 
    When $2R^2_D(p') \leq \min \{ R_{c^{1}}^2, R_{c^{2}}^2 \}$, 
    \begin{align*}
        \sup_{r_1, r_2 \in [0, R^2_D(p')], \theta \in [0, 2\pi)}  \sigma_{\mc{E}}(r_1, r_2, \theta)  \leq & 2 \bigg( h\big( 0 \big) - h\big( 8R^2_D(p') \big) \bigg) 
    \end{align*}
\end{corollary}

\begin{proposition} [Maximization over the Far Regime]\label{proposition: maximazation_far_regime}
    When $r_1^2 + r_2^2 \geq \max \{ R_{c^{1}}^2, R_{c^{2}}^2 \}$, 
    \begin{align*}
        \sup_{r_1, r_2 \in [0, R^2_D(p')], \theta \in [0, 2\pi)}  \sigma_{\mc{E}}(r_1, r_2, \theta)  \leq & 2 \bigg( h\big( 0 \big) - h\big( 8R^2_D(p') \big) \bigg) +4 \cdot \bigg( -h'\big( 0 \big) + 2h'\big( 2R^2_D(p') \big) \bigg) \cdot \bigg(4R^2_D(p') \bigg)  \\
        &+ 4 \cdot \bigg(h''\big( 0 \big) - 2h''\big( 2R^2_D(p') \big) \bigg) \cdot \bigg(4R^4_T(p') \bigg) 
    \end{align*}
\end{proposition}
\begin{proof}
    Similarly, we first fix $ r_1$ and $r_2$ and consider only $\theta$. \\
    In this regime, both Term 2a and Term 3a are positive.
    Hence, Term 2 is maximized when Term 2b is maximized, which occurs at $\cos \theta = -1$, i.e., $\theta = \pi$.
    Similarly, Term 3 is maximized when Term 3b is maximized, which occurs at $\cos^2 \theta = 0$, i.e., $\theta = \frac{\pi}{2}$.
    Then
    \begin{align*}
        \sup_{r_1, r_2 \in [0, R^2_D(p')]} \sigma_{\mc{E}}(r_1, r_2, \theta)  
        &\leq \sup_{r_1, r_2 \in [0, R^2_D(p')]} \bigg[ \tx{Term 1 at } \theta = \pi + \tx{Term 2 at } \theta = \pi + \tx{Term 3 at } \theta = \pi / 2 \bigg] \\  
        &= \sup_{r_1, r_2 \in [0, R^2_D(p')]} 2 \bigg( h\big( 0 \big) - h\big( 2(r_1 + r_2)^2 \big) \bigg) +4 \cdot \bigg( -h'\big( 0 \big) + 2h'\big( r_1^2 + r_2^2 \big) \bigg) \cdot \bigg(r_1 + r_2 \bigg)^2  \\
        &\quad \quad + 4 \cdot \bigg(h''\big( 0 \big) - 2h''\big( r_1^2 + r_2^2 \big) \bigg) \cdot \bigg(r_1^2 + r_2^2 \bigg)^2 \\
    \end{align*}
    
    Term 2 and Term 3 are both positive, hence Term 2a, Term 2b, Term 3a, and Term 3b are all increasing with respect to $r_1, r_2$.
    Hence, the supremum is attained when $r_1^2 = r_2^2 = R^2_D(p')$.
\end{proof}

\begin{remark}
    The upper bounds for Local and Far regimes may have different degrees of tightness.
\end{remark}

\subsection{Uniform Error Bound for Randomly Distributed Inputs}
In practice, inputs to the random field are often drawn from a probability distribution $\pi$.
We seek a reasonably tight upper bound that holds with high probability $1-\delta$.

\begin{definition} [Uniform Upper Bound for Randomly Distributed Inputs] 
\begin{align*}
        &b^{*}_\delta(\pi) := inf \cd{b \in \mb{R}:  \mb{P}\dc{ \mathcal{E}(\xo, \xt) \leq b} > 1 - \delta \tx{ where } \xo, \xt \iid \pi }  \label{definition: uniform_bound_delta} 
    \end{align*}
\end{definition}

However, this rarely admits concise analytic expressions. Alternatively, we consider the infimum among the maximal error values over every tolerance region that satisfies the confidence level $p$:
\begin{equation} \label{definition:uniform_bound_pp}
    b^{*}_{(p,p')}(\pi) := \inf_{A} \bigg\{\sup_{\xo, \xt \in A} b_p^{*} (\xo, \xt) : A \subseteq \mb{R}^D \tx{ satisfies } \mb{P} \left( \xo, \xt \in A \mid \xo, \xt \iid \pi \right) > p' \bigg\} \tx{  } 
\end{equation}
\begin{notation}[Highest Density Region among all Admissible Regions]
    If it exists, we denote the smallest region that satisfies the constraint as follows:
    \begin{align*}
        A_{p'}(\pi) := \bigcap \bigg\{ A: \mb{P} \left( \xo, \xt \in A \mid \xo, \xt \iid \pi \right) > p' \bigg\}
    \end{align*}
\end{notation}

The proposition below shows the relationship between the two notions of uniform error bound. Our approach to deriving the uniform error bound is based on this connection. It restricts our concern to a maximization problem over a fixed region $A_{p'}(\pi)$.

\begin{proposition} \label{prop:uniform_error_bound_relationship}
For any  $p,p' \in [0,1]$ satisfying $\delta = p \cdot p'$, we have
\begin{align*}
b^{*}_{\delta}(\pi) \leq b^{*}_{(p,p')}(\pi) \tx{ where }  \pi \tx{ is a probability distribution over }  \mb{R}^{D}
\end{align*}  
\end{proposition}
\begin{proof}
    Let $p, p' \in [0,1]$ satisfy $\delta = p \cdot p'$ and choose a set $A \subseteq \mathbb{R}^d$ that satisfies the following conditions:
\begin{align*}
    \mathbb{P} \left[ \mathcal{E}(\xo, \xt) \leq b \mid \xo \in A, \xt \in A \right] &> p ,\quad  \mathbb{P} \left[ \xo \in A, \xt \in A \mid \xo, \xt \overset{\text{i.i.d.}}{\sim} \pi \right] > p'
\end{align*}
Then
\begin{align*}
\mb{P}\dc{ \mathcal{E}(\xo, \xt) \leq b \mid \xo, \xt \iid \pi }  &\geq \mb{P}\dc{ \mathcal{E}(\xo, \xt) \leq b \mid \xo, \xt \in A } \cdot \mb{P}\dc{ \xo, \xt \in A \mid \xo, \xt \iid \pi} \\
&> p \cdot p' = \delta.
\end{align*}
So, $b^{*}_{(p,p')}(\pi)$ satisfies
\begin{align*}
\mb{P} \left[\mathcal{E}(\xo, \xt) \leq b^{*}_{(p,p')}(\pi) \right] > \delta \tx{ where } \xo, \xt \iid (\pi)
\end{align*}
Since $b^*_{\delta}(\pi)$ is defined as the infimum of all $b$ satisfying the condition above, it follows that $b^*_{\delta}(\pi) \leq b^{*}{(p,p')}(\pi)$.
\end{proof}

Suppose the inputs are drawn from a multivariate normal distribution, i.e., $\xo, \xt \iid \mc{N}(z, \Sigma)$ for some covariance matrix $\Sigma \in \mb{R}^{D \times D}$. 
When  $\xo, \xt \iid \mc{N}(z, \Sigma)$, $(x-z)^{\tr}\Sigma^{-1}(x-z)$ follows Chi-squared distribution with $D$ degrees of freedom. Then, since $\xo, \xt$ are independently sampled, the region $A_{p'}(\pi)$ is given by a concentration ellipsoid as follows:
\begin{align*}
    A_{p'}(\pi) &:= \bigg\{ x \in \mb{R}^D: x^\tr \Sigma^{-1} x \leq F^{-1}_{\chi^2_D}(\sqrt{p'}) \bigg\} \tx{ where }  \label{equation:normal_concentration_ellipsoid}\\
    &F^{-1}_{\chi^2_D}(\cdot) \tx{ is the quantile function of the Chi-squared distribution with degree of freedom D}. 
\end{align*}

% Then, for any $p' \in (0,1)$, the set 
% satisfies 
% \begin{align*}
%     \mb{P}\dc{\xo, \xt \in A_{p'} \mid \xo, \xt \iid \mc{N}(z, \Sigma)} > p'
% \end{align*}

In many instances, the center of the normal distribution represents the critical point, and this is where the local analysis is conducted.
In particular, for skew-symmetric function modeling, we focus on a multivariate normal distribution centered at matched inputs, i.e., $z = (\zz, \zz)$, due to the homogeneity among the component spaces.
This allows us to construct the uniform upper bound $b_{p'}(\pi)$ based on the maximal error results introduced in Section \ref{section: maximal_error_over_ellipsoid}.

\begin{theorem}
    For any pair $p, p' \in (0,1)$ satisfying $\delta = p \cdot p'$, the uniform error bound $b^{*}_{\delta}$ is upper bounded as follows:
    \begin{align*}
        \begin{cases}
        & b^*_{\delta} \leq  Q(p) \cdot \bigg[ 2 \bigg( h\big( 0 \big) - h\big( 8R^2_D(p') \big) \bigg) \bigg] \quad \text{ if } 2R^2_D(p') \leq \min \{ R_{c^{1}}^2, R_{c^{2}}^2 \}\\
         \\
        & b^*_{\delta} \leq Q(p) \cdot \bigg[ 2 \bigg( h\big( 0 \big) - h\big( 8R^2_D(p') \big) \bigg) +4 \cdot \bigg( -h'\big( 0 \big) + 2h'\big( 2R^2_D(p') \big) \bigg) \cdot \bigg(4R^2_D(p') \bigg)  \\
        &\qquad\qquad\qquad\qquad \ldots   +4 \cdot \bigg(h''\big( 0 \big) - 2h''\big( 2R^2_D(p') \big) \bigg) \cdot \bigg(4R^4_T(p') \bigg) \bigg] \quad \text{ otherwise.}
    \end{cases}
    \end{align*}
\end{theorem}
\begin{proof}
    Combine Propositions \ref{proposition: maximazation_local_regime},  \ref{proposition: maximazation_far_regime},\ref{prop:uniform_error_bound_relationship} and the definition of $b_{(p,p')}$ defined in \ref{definition:uniform_bound_pp}.
\end{proof}

\subsection{Asymptotic Behavior of the Error Bound for Small and Large Regions}
To summarize the uniform upper bound's behavior for small and large regions, analyze its rate of decay as the radius of the region vanishes (local) and its rate of growth as the radius diverges (far).
Our characterization is based on the properties of the function $h: \mb{R}_{\geq0} \to \mb{R}$, provided in Notation \ref{notation:kernel_function_h}.
For the local regime analysis, we rely on the smoothness $h \in C^2(\mb{R}_{\geq0})$. For the far regime counterpart, the monotonicity $(-1)^{n}h^{(n)}(\cdot) \geq 0$ for $n=0,1,2$ is used.

\begin{notation} 
We denote the expression for the upper bound with respect to $R >0$, which stands for the range of the $r_1, r_2>0$ (distance between $z^*$ and each point $\xo, \xt \in \mb{R}^D$).
$\phi, \psi: \mb{R}_{\geq0} \to \mb{R}$ are used for the expression on the local and far regimes:
\begin{align*}
    &\textbf{Local: } \phi(R) = 2 \bigg( h\big( 0 \big) - h\big( 8R^2 \big) \bigg)\\    
    &\textbf{Far: } \hspace{0.08 in} \psi(R) =  2 \bigg( h\big( 0 \big) - h\big( 8R^2 \big) \bigg) +4 \cdot \bigg( -h'\big( 0 \big) + 2h'\big( 2R^2 \big) \bigg) \cdot \bigg(4R^2 \bigg) + 
    4 \cdot \bigg(h''\big( 0 \big) - 2h''\big( 2R^2 \big) \bigg) \cdot \bigg(4R^4 \bigg) 
\end{align*}
\end{notation}

\begin{proposition} [Rate of Decay of Uniform Error Bound] \label{proposition:rate_of_decay_bound}
The uniform upper bound decays quadratically, with constant determined by the first derivative of $h$ at zero. Specifically:
\begin{align*}
\phi(R) \to -16h'(0)R^2 + O(R^4) \tx{ as } R \to 0
\end{align*}
\end{proposition}
\begin{proof}
    Second order Taylor expansion of the function $h$ at $z=0$ is given by: $h(0) + h'(0) t + (1/2)h''(0)t^2$.
    Consequently, $\phi(R) = 2\bigg(h(0) - h(8R^2) \bigg)\simeq 2\bigg( -8h'(0)R^2 - 32h''(0)R^4\bigg)$.
\end{proof}

\begin{proposition} [Limiting Behavior of Uniform Error Bound] \label{proposition:limiting_behavior_bound}
At infinity, the upper bound grows quartically, with the constant determined by the second derivative of $h$ at zero. %Notably, the kernel value at $h(0)$ remains independent of the distance. To be specific:
Specifically:
\begin{align*}
\psi \to  2 h\big( 0 \big)  - 16 h'\big( 0 \big) R^2 + 16 h''\big( 0 \big)R^4  \tx{ as } R \to \infty
\end{align*}
\end{proposition}
\begin{proof}
    For $h$ included in the function class described as above, $h(8R^2), h'(2R^2), h''(2R^2) \to 0 \tx{ as } R\to 0$.
    From this, the conclusion follows immediately.
\end{proof}

Theorem \ref{thm:Covariance_structure} provides clear interpretations for the $h(0)$, $h'(0)$, and $h''(0)$ terms:
\begin{align*}
    &h(0) = \text{Var} \bigg(u(z)\bigg) \\
    &h'(0) = -\frac{1}{4} \text{Var} \bigg( \dro_i u(z)\bigg) = -\frac{1}{4} \text{Var} \bigg( g(z^*,z^*)\bigg) \tx{ } \tx{ for some } i = 1, \dots, D \\
    &h''(0) = \frac{1}{8}\text{Var}(\droo_{i,j} u(z)) = \frac{1}{8}\text{Var}(H^{(1,1)}(z^*,z^*)) \tx{ }  \tx{ for some } i\neq j = 1, \dots, D
\end{align*}
These control, respectively, the variance in the process, the variance in the linear part of its approximation, and the variance in the quadratic part of its approximation.

In summary:
\begin{enumerate}
\item In the local regime, the upper bound decays quadratically, and its magnitude is determined by the variance of the slope of the base field near the point of expansion $z \in \mathbb{R}^D$. The error bound is larger when the slope of the base field is more variable.
\item In the far regime, grows quartically, and its magnitude is determined by the variance of the second derivatives of the base field near the point of expansion $z \in \mathbb{R}^D$. The error bound is larger when the second derivatives of the base field are more variable.
\end{enumerate}
Note that $h(0)$ appears in the limit $R \to \infty$, but vanishes in the limit $R \to 0$. This observation is consistent with the behavior of generic local approximations to random fields whose entries are approximately independent at large distances. Locally, the quadratic approximation can replace the original random function, with the residual vanishing quadratically. When $R = 0$, the approximation is exact, so the variance in the process does not appear in the error on small neighborhoods. In the far regime, $f(x,y)$ is essentially independent of $f(z+v,z + w)$ for any small $v,w$, so is independent of the derivatives used to generate the quadratic approximation $\tilde{f}$. Then, the variances add as if the process, and approximation, were sampled independently.

\subsubsection*{Example: Squared Exponential Kernel}
The function $h(\tau) = \exp(-\tau/l^2)$ defines the squared exponential covariance function $k_u(x,y) = \exp\big(-|x-y|_2^2/l^2 \big)$ (see \ref{notation:kernel_function_h}).
The length scale parameter $l > 0$ admits a straightforward interpretation. If $l$ is small, then the covariance function decays rapidly. If $l$ is large, then the covariance function deacys slowly, and sampled values remain correlated at long distances.

The derivatives of $h$ at the origin are $h(0) = 1$, $h'(0) = -1/l^2$, and $h''(0) = 1/l^4$.
Small $l$ produce large $h'(0)$ and $h''(0)$, which implies a high variance in local curvature.
So, shrinking the length scale $l$ increases the variability in the terms that parameterize the quadratic approximation, while reducing the distances needed to decorrelate sampled values of the process, producing an error bound that decays slowly as $R$ approaches zero, and grows quickly as $R$ diverges.
%\begin{enumerate}
%\item To obtain a small error, the locations must be very close to the point of the Taylor expansion.
%\item As soon as two locations become distant from the expansion point, the local approximation becomes highly inaccurate.
%\end{enumerate}

\subsection{Asymptotic Analysis with respect to the Dimension}
In high-dimensional settings, probabilistic results may yield surprising conclusions \cite{vershynin2009high}. Accordingly, it is important to directly study how systems behave in the limit of large dimension. Here, we will investigate how the accuracy of quadratic approximations to Gaussian fields does, or does not, depend on the ambient dimension. 

We've characterized the accuracy of the approximation with an error bound that holds with high probability over elliptical regions, and for Gaussian distributed inputs. Our construction involves two sources of randomness: one from randomly sampled from Gaussian field and the other from random inputs distributed according to the probability distribution $\pi$. The construction depends on three interrelated factors: the confidence level $\delta \in [0,1]$, the desired upper bound $b = b^*_{\delta}(\pi)$, and the admissible region $A = A_{p'}(\pi)$. Each choice of these parameters leads to the following question:
\begin{enumerate}
\item For fixed $\delta$ and $b$, how does the admissible region vary?
\item For fixed $\delta$ and $A$, how does the uniform upper bound vary?
\item For fixed $A$ and $b$, how does the confidence level vary?
\end{enumerate}
These are, in effect, equivalent questions. We illustrate the influence of the dimension $D$ by focusing on the first.
Contextually, this question can be phrased as follows: \textit{To keep a high-probability bound on the worst-case error constant as the dimension $D$ grows, how must the admissible region vary? }

If the admissible must shrink rapidly in dimension to maintain the same error control, then quadratic approximations will lose fidelity on neighborhoods of any fixed size, for sufficiently high-dimensional problems. Then, to employ a quadratic approximation in high dimension, an analyst would either need a tighter error bound, a different notion of error, or stronger regularity conditions on the process. In contrast, if the admissible region's size is independent of the dimension, then we can shift our focus to guaranteeing that inputs stay within the region in high dimension. %At a high level, the curse-of-dimensionality argument suggests that the concentration ellipsoid shrinks as the dimensionality increases. This intuition can be justified by tracing back our construction.

Our construction of the uniform bound was based on Proposition \ref{proposition: hierarchy_of_maximization}. Proposition \ref{proposition: hierarchy_of_maximization}, allowed, without loss of generality, reduction from $D$ dimensions to a standard two-dimensional optimization problem. As a result, the pointwise standard deviation in the error process can be bounded above by a constant that is independent of the dimension $D$ for any compact region $A$. When $A$ is contained in an ellipsoid, the bound depends only on the lengths of the first and second principal axes of the ellipsoid. It follows that, if we are only concerned with worst-case error bounds on fixed neighborhoods, then we do not need to vary the neighborhood's radius to account for changes in the dimension $D$. In particular, we can use the same radius, even as $D$ diverges. %the admissible region is represented in two-dimensional Euclidean space by the circular region $R^2_D(p') = \lambda_1 \cdot \chi^2_D(\sqrt{p'})$. 

This optimistic result is countered by the observation that, the volume of a $D$ dimensional ball vanishes in the limit as $D$ diverges. So, while the worst-case error is bounded above by a bound that holds with high probability, for a region with fixed radius, the neighborhood may need to vary in $D$ if we sample inputs at random about the point of expansion, and hope to keep the same fraction of inputs inside the admissible region.

When $\pi$ is Gaussian, the admissible region may be represented in two-dimensional Euclidean space by the circular region $R^2_D(p') = \lambda_1 \cdot \chi^2_D(\sqrt{p'})$. This implies that the geometric constraint imposed by a fixed confidence level and a fixed upper bound restricts only the magnitude of the principal axis; that is, the largest eigenvalue of $\Sigma$.
The uniform upper bound remains constant at a fixed confidence level only if we consider a “smaller” admissible region, namely one for which the largest eigenvalue of the covariance matrix satisfies
\begin{align*}
    \lambda_1 \propto \big( \chi^2_D(\sqrt{p'}) \big)^{-1}
\end{align*}
To characterize this trend, we introduce a simple closed-form approximation to the quantile function of the chi-squared distribution, originally proposed by R.A Fisher in \cite{fisher1970statistical}. Among various higher-precision approximations, this one yields an especially simple expression while remaining accurate for degrees of freedom greater than 30, thereby providing a convenient tool for understanding the overall behavior. Specifically, Fisher’s approximation for the $p$-quantile of the $\chi^2_D$ distribution is
\begin{equation} 
\chi^2_D(p) \approx \frac{1}{2} \left( \Phi^{-1}(p) + \sqrt{2D - 1} \right)^2.
\end{equation}
Using Fisher’s approximation, we conclude that: to maintain the same confidence level for randomly selected inputs and the same uniform upper bound on the error as the dimension grows, the admissible ellipsoid must shrink, with the magnitude of its principal axis decreasing at the rate $O(D^{-1})$. This result is natural since, the expected squared length of a Gaussian random vector is proportional to the dimension $D$.

%%%%%%%%%%%%%%%%%%%%%%%%%%%%%%%%%%%%%%%%%%%%%%%%%%%%%%%%%%%%%%%%%%%%%%%%%%%%%%%%%%%%%%%%%%%%%%%%%%%%%%%%%%%%%%%%%%%%%%%%%%%%%%%%%%%%%%%%%%%%%%%%%%%%%%%%%%
% \textit{Overall, this draft is in a good place. The first two thirds are wonderfully tight.}

% \textit{The last section on error bounds needs attention. You need to remind the reader why we are solving this problem, and make it clearer why we needed each step in the analysis. It would help to end with clearer conclusions (e.g. how does the uniform bound scale as a function $R_D$, and how would we choose $R_D(D)$ to keep the bound constant as $D$ increases). To answer the first question for small $R_D$, Taylor expand $h(s)$, $h'(s)$, and $h''(s)$ about zero, keeping only the lowest order terms. For large $R_D$ your monotonicity and sign requirements force all of the $h(s)$, $h'(s)$, $h''(s)$ terms to zero. Eyeballing, it looks like the error bound grounds quadratically in the radius of the region when the radius is small, then quartically when the radius is large.}

%To assess its accuracy, we plot the exact critical values and Fisher’s approximation, together with their reciprocals, as functions of the degrees of freedom $D$ for $p' = 0.95$. This can be found in appendix \ref{figure:chi_quantile_figures}. 

%% file: Conclusion.tex
\section{Conclusion} \label{sec: conclusion}

This manuscript provides a basic distributional characterization of local quadratic approximations to skew-symmetric Gaussian fields and their errors. We've established that the first and second order derivatives of a block-isotropic Gaussian are independent, and that the Hessian breaks into GOE and GSOE blocks. We have also shown that the worst-case error in the quadratic approximation on an elliptical region can be bounded from above, with high probability, by a bound that is quadratic in the radius of the region for small radii, then quartic for large radii. The rate of growth in the bound depends on the length scale of the kernel function, and, as a consequence, the smoothness of the process. Naturally, rougher processes produce larger errors for the same radii than smoother processes. Importantly, the maximal approximation error over a compact region is bounded from above with high probability by a bound whose value is independent of the underlying dimension. When the inputs are drawn at random from a multivariate normal, the error bound depends on the dimension only through the sampling of inputs to the approximation. 

We propose this work as a reference for future studies based on local approximations to random, skew-symmetric functions.

%% file: appendix.tex
\section{Derivation of Entrywise Covariance Structures} \label{appendix:Covariance_derivations}
Throughout the derivations, we use the following notations, which stand for the squared norms of coordinate differences:
\begin{align*}
  p(x,y) &= \| x - y \|_2^2 = \sum_{m=1}^D (\xo_m - \yo_m)^2 + (\xt_m - \yt_m)^2 \\
  q(x,y) &= \| x - \yt, y - \xo \|_2^2 = \sum_{m=1}^D (\xo_m - \yt_m)^2 + (\xt_m - \yo_m)^2
\end{align*}

Since we consistently reserve $ x$ and $y$ for their own role in the context, we omit the arguments and write $p$ and $q$ when there is no ambiguity.\\
To begin with, we aggregate some facts that will be repeatedly used in the derivations.
The derivatives of coordinate differences are given by:
\begin{align*}
  &\partial_{\xo_i} p = 2(\xo_i - \yo_i) \quad \partial_{\xt_i} p = 2(\xt_i - \yt_i) \quad \quad &&\partial_{\xo_i} q = 2(\xo_i - \yt_i) \quad \partial_{\xt_i} q = 2(\xt_i - \yo_i) \\
  &\partial^2_{\xo_i \xo_j} p = 2 \delta_{ij} \quad \partial^2_{\xt_i \xt_j} p = 2 \delta_{ij}  \quad \quad &&\partial^2_{\xo_i \xo_j} q = 2 \delta_{ij} \quad \partial^2_{\xt_i \xt_j} q = 2 \delta_{ij}\\
  &\tx{While mixed partial derivatives vanish: }\\
  &\partial^2_{\xo_i \xt_j} p = 0  \quad &&\partial^2_{\xo_i \xt_j} q = 0\\
  &\tx{When we take second derivatives with respect to } y: \\
  &\partial_{\yo_k}\bigg(\partial_{\xo_i} \bigg) p =  -2 \delta_{ik} \quad \partial_{\yt_k}\bigg(\partial_{\xt_i} \bigg) p = -2 \delta_{ik}  \quad \quad &&\partial_{\yt_k}\bigg(\partial_{\xo_i} \bigg) q = -2 \delta_{ik} \quad \partial_{\yo_k}\bigg(\partial_{\xt_i} \bigg) q = -2 \delta_{ik}\\
  &\tx{While mixed partial derivatives vanish again: }\\
  &\partial_{\yt_k}\bigg(\partial_{\xo_i} \bigg) p = 0  \quad &&\partial_{\yo_k}\bigg(\partial_{\xt_i} \bigg) q = 0\\
\end{align*}
The first derivatives evaluated at $x = z$ but $y$ arbitrary does not necessarily vanish:
\begin{align*}
  &\partial_{\xo_i} p \big|_{x=z} = 2(\zz_i - \yo_i), \quad \partial_{\xt_i} p \big|_{x=z} = 2(\zz_i - \yt_i) \quad \quad \partial_{\xo_i} q \big|_{x=z} = 2(\zz_i - \yt_i), \quad \partial_{\xt_i} q \big|_{x=z} = 2(\zz_i - \yo_i)
\end{align*}
In comparison, the same quantities always vanish as long as $x = y$.
\begin{align*}
  &\partial_{\xo_i} p \big|_{x=y} = 0, \quad \partial_{\xt_i} p \big|_{x=y} = 0, \quad \quad  \partial_{\xo_i} q \big|_{x=y} = 0, \quad \partial_{\xt_i} q \big|_{x=y} = 0
\end{align*}

The covariances we need to calculate can be distinguished by the differentials taken at $x$. To be specific, there are two cases:

\begin{align*}
  \begin{cases}
    \tx{Case 1: $x = z = (\zz, \zz)$ but $y = (\yo, \yt)$ is arbitrary } \\
    \tx{Case 2: $x = y = z = (\zz, \zz)$}
  \end{cases}
\end{align*}

\subsection{Case 1: $x = z = (\zz, \zz)$ but $y = (\yo, \yt)$ is arbitrary}
Note that 
\begin{align*}
  p(z,y) &= \sum_{m=1}^D (\zz_m - \yo_m)^2 + (\zz_m - \yt_m)^2 = ||z - y||_2^2 \\
  q(z,y) &=  \sum_{m=1}^D (\zz_m - \yt_m)^2 + (\zz_m - \yo_m)^2 = ||z - y||_2^2 \\
\end{align*}

The following covariances fall into this case, and, under the smoothness assumptions on $h$, they equal the right-hand sides shown below.

\begin{align*}
&\Cov\bigg( \dro_i f(\zz, \zz), f(\yo,\yt)\bigg) = \der_{\xo_i} k_f \bigg( (\xo, \zz), (\zz, \zz) \bigg) \bigg|_{\xo=\zz}\\
&\Cov\bigg( \dro_{i,j} f(\zz, \zz), f(\yo,\yt)\bigg) =  \der^2_{\xo_i \xo_j} k_f \bigg( (\xo, \zz), (\zz, \zz) \bigg) \bigg|_{\xo=\zz}\\
&\Cov\bigg( \drt_{i,j} f(\zz, \zz), f(\yo,\yt)\bigg) = \der^2_{\xo_i \xt_j} k_f \bigg( (\xo, \xt), (\zz, \zz) \bigg) \bigg|_{\xo=\zz, \xt=\zz}\\
\end{align*}
Recall that 
\begin{align*}
  k_f(x,y) = 2 \bigg[ h(p) - h(q) \bigg]
\end{align*}
Then it follows that
\begin{align*}
\der_{\xo_i} k_f(x,y) &= 2 \bigg[ \der_{\xo_i} p \cdot h'(p) - \der_{\xo_i} q \cdot h'(q) \bigg] \\
\der^2_{\xo_i \xo_j} k_f(x,y) &= 2 \bigg[ \der^2_{\xo_i \xo_j} p \cdot h'(p) + \der_{\xo_i} p \cdot \der_{\xo_j} p \cdot h''(p) - \der^2_{\xo_i \xo_j} q \cdot h'(q) - \der_{\xo_i} q \cdot \der_{\xo_j} q \cdot h''(q) \bigg] \\
\der^2_{\xo_i \xt_j} k_f(x,y) &= 2 \bigg[ \der^2_{\xo_i \xt_j} p \cdot h'(p) + \der_{\xo_i} p \cdot \der_{\xt_j} p \cdot h''(p) - \der^2_{\xo_i \xt_j} q \cdot h'(q) - \der_{\xo_i} q \cdot \der_{\xt_j} q \cdot h''(q) \bigg] \\
\text{Evaluated at } x = z, &\tx{ we have } \\
\der_{\xo_i} k_f(z,y) &= 2 \bigg[ 2(\zz_i - \yo_i) \cdot h'(p) - 2(\zz_i - \yt_i) \cdot h'(q) \bigg] \\
  &= 4 h'\bigg(||z - y||_2^2\bigg) \bigg[(\zz_i - \yo_i) - (\zz_i - \yt_i))\bigg] \\
\der^2_{\xo_i \xo_j} k_f(z,y) &= 2 \bigg[ 2 \delta_{ij} \cdot h'(p) + 4(\zz_i - \yo_i)(\zz_j - \yo_j) \cdot h''(p) - 2 \delta_{ij} \cdot h'(q) - 4(\zz_i - \yt_i)(\zz_j - \yt_j) \cdot h''(q) \bigg] \\
&= 8 h''\bigg(||z - y||_2^2\bigg) \bigg[ (\zz_i - \yo_i)(\zz_j - \yo_j) - (\zz_i - \yt_i)(\zz_j - \yt_j) \bigg] \\
\der^2_{\xo_i \xt_j} k_f(z,y) &= 2 \bigg[ 0 \cdot h'(p) + 4(\zz_i - \yo_i)(\zz_j - \yt_j) \cdot h''(p) - 0 \cdot h'(q) - 4(\zz_i - \yt_i)(\zz_j - \yo_j) \cdot h''(q) \bigg] \\
&= 8 h''\bigg(||z - y||_2^2\bigg) \bigg[ (\zz_i - \yo_i)(\zz_j - \yt_j) - (\zz_i - \yt_i)(\zz_j - \yo_j) \bigg] \\
\end{align*}

\subsection{Case 2: $x = y = z = (\zz, \zz)$}
Note that
\begin{align*}
  p(z,z) &= \sum_{m=1}^D (\zz_m - \zz_m)^2 + (\zz_m - \zz_m)^2 = 0 \\
  q(z,z) &=  \sum_{m=1}^D (\zz_m - \zz_m)^2 + (\zz_m - \zz_m)^2 = 0 \\
\end{align*}
Under this, there should be
\paragraph{Covariances between each pair of $g, h_{xx}, h_{xy}$ evaluated at $z$.}

\begin{align*}
\Cov\bigg(\hoo_{i,j}, g_k \bigg) &:= \Cov\bigg( \droo_{i,j} f(\zz, \zz), \dro_k f(\zz, \zz) \bigg) \\
 &= \der^3_{\xo_i \xo_j \yo_k} k_f \bigg( (\xo, \zz), (\yo, \zz) \bigg) \bigg|_{\substack{\xo=\zz \\ \yo=\zz}} = \der^3_{\xo_i \yo_k  \xo_j} k_f \bigg( (\xo, \zz), (\yo, \zz) \bigg) \bigg|_{\substack{\xo=\zz \\ \yo=\zz}} \\
\Cov\bigg(\hot_{i,j}, g_k \bigg) &:= \Cov\bigg( \drot_{i,j} f(\zz, \zz), \dro_k f(\zz, \zz) \bigg) \\
 &= \der^3_{\xt_j \yo_k \xo_i  } k_f \bigg( (\xo, \xt), (\yo, \yt) \bigg) \bigg|_{\substack{\xo=\zz  \xt=\zz \\ \yo=\zz  \yt=\zz}} \\
\Cov\bigg(\hot_{i,j}, \hoo_{k,l} \bigg) &:= \Cov\bigg( \drot_{i,j} f(\zz, \zz), \droo_{k,l} f(\zz, \zz) \bigg) \\
 &= \der^4_{\yo_l \xo_i \xt_j \yo_k  } k_f \bigg( (\xo, \xt), (\yo, \zz) \bigg) \bigg|_{\substack{\xo=\zz  \xt=\zz \\ \yo=\zz}} 
\end{align*}

Recall that
\begin{align*}
  \der_{\xo_i} k_f(x,y) &= 2 \bigg[ \der_{\xo_i} p \cdot h'(p) - \der_{\xo_i} q \cdot h'(q) \bigg] 
\end{align*}
Then it follows that
\begin{align*}
\der_{\yo_k \xo_i} k_f(x,y) &= 2 \bigg[ \der_{\yo_k}\bigg(\der_{\xo_i} p \bigg) \cdot h'(p) + \der_{\xo_i} p \cdot \der_{\yo_k} p \cdot h''(p) - \der_{\yo_k}\bigg(\der_{\xo_i} q \bigg) \cdot h'(q) - \der_{\xo_i} q \cdot \der_{\yo_k} q \cdot h''(q) \bigg] \\
\tx{Evaluated at } y = z, &\tx{ every coordinate difference vanishes }. \tx{ Thus, } \\
\der_{\yo_k \xo_i} k_f(x,z) &= 2 \bigg[ -2 \delta_{ik} \cdot h'(p) - (-2 \delta_{ik}) \cdot h'(q) \bigg] = 0
\end{align*}
This implies that 
\begin{align*}
  \der_{\xt_j} \bigg( \der^2_{\yo_k \xo_i  } k_f \bigg( (\xo, \xt), (\yo, \yt) \bigg) \bigg|_{\substack{\xo=\zz \\ \yo=\zz  \yt=\zz}}\bigg) \bigg|_{\xt = \zz}  = \Cov\bigg(\hoo_{i,j}, g_k \bigg) = 0
\end{align*}
Similarly, from $\der_{\xt_j} k_f(x,y) = 2 \bigg[ \der_{\xt_j} p \cdot h'(p) - \der_{\xt_j} q \cdot h'(q) \bigg]$, we have
\begin{align*}
\der_{\yo_k \xt_j} k_f(x,y) &= 2 \bigg[ \der_{\yo_k}\bigg(\der_{\xt_j} p \bigg) \cdot h'(p) + \der_{\xt_j} p \cdot \der_{\yo_k} p \cdot h''(p) - \der_{\yo_k}\bigg(\der_{\xt_j} q \bigg) \cdot h'(q) - \der_{\xt_j} q \cdot \der_{\yo_k} q \cdot h''(q) \bigg] \\
&= 2 \bigg[ \der_{\xt_j} p \cdot \der_{\yo_k} p \cdot h''(p) - \der_{\xt_j} q \cdot \der_{\yo_k} q \cdot h''(q) \bigg] \\
&= 2 h''(0) \cdot 4 \cdot \bigg[ (\xt_j - \yt_j)\cdot (-(\xo_k - \yo_k )) - (\xt_j - \yo_j) \cdot (-(\xt_k - \yo_k))  \bigg] \\
\end{align*}
Evaluating at $x = y = z$  but $\yo_l, \xo_i$, all the coordinate differences vanish.
This implies that
\begin{align*}
  \Cov\bigg(\hot_{i,j}, g_k \bigg) = 0
\end{align*}

\paragraph{Covariances between $g, h_{xx}, h_{xy}$ themselves, respectively, evaluated at $z$ }
% \begin{align*}
%   Cov\bigg( g_i, g_k \bigg) &:= Cov\bigg( \dro_i f(\zz, \zz), \dro_k f(\zz, \zz) \bigg) = \der^2_{\xo_i \yo_k} k_f \bigg( (\xo, \zz), (\yo, \zz) \bigg) \bigg|_{\substack{\xo=\zz \\ \yo=\zz}}\\
%   Cov\bigg( h_{xx_{i,j}}, h_{xx_{k,l}} \bigg) &:= Cov\bigg( \droo_{i,j} f(\zz, \zz), \droo_{k,l} f(\zz, \zz) \bigg) \\
%   &= \der^4_{\xo_i \xo_j \yo_k \yo_l} k_f \bigg( (\xo, \zz), (\yo, \zz) \bigg) \bigg|_{\substack {\xo=\zz \\ \yo=\zz}} = \der^4_{\yo_k \yo_l \xo_i \xo_j} k_f \bigg( (\xo, \zz), (\yo, \zz) \bigg) \bigg|_{\substack{\xo=\zz \\ \yo=\zz}} \\
%   Cov\bigg( h_{xy_{i,j}}, h_{xy_{k,l}} \bigg) &:= Cov\bigg( \drot_{i,j} f(\zz, \zz), \drot_{k,l} f(\zz, \zz) \bigg) \\
%   &= \der^4_{\xo_i \xt_j \yo_k \yt_l} k_f \bigg( (\xo, \xt), (\yo, \yt) \bigg) \bigg|_{\substack{\xo=\zz  \xt=\zz \\ \yo=\zz  \yt=\zz}}
% \end{align*}
Similarly, we can show that
\begin{align*}
  &\Cov\bigg( g_i, g_k \bigg) = -4h'(0)\cdot  \delta_{ik} \\
  &\Cov\bigg( [H_{xx}]_{i,j}, [H_{xx}]_{k,l} \bigg) = 8 h''(0) \bigg( \delta_{ik} \delta_{jl} + \delta_{il} \delta_{jk} \bigg) \\
  &\Cov\bigg( [H_{xy}]_{i,j},[H_{xy}]_{k,l}  \bigg) = 8 h''(0) \bigg( \delta_{ik} \delta_{jl} - \delta_{il} \delta_{jk}\bigg)
\end{align*}

\section{Details on the proposition \ref{prop:Covariance_function_Taylor_approximation} (Covariance Functions of $\tilde{f}$) } \label{appendix:Covariances_interacting_terms}
\subsection{Derivation of Covariance between Interacting Terms}
\begin{align*}
    &\Cov\big( g^\tr \big( \vo - \vt \big), g^\tr \big( \too - \ttt \big) \big) \\
    &= \Cov\bigg( \sum_{i=1}^{D} g_i (\vo_i - \vt_i), \sum_{j=1}^{D} g_j (\too_j - \ttt_j) \bigg) \\
    &= \sum_{i=1}^{D} \sum_{j=1}^{D} (\vo_i - \vt_i)(\too_j - \ttt_j) \Cov(g_i, g_j) \\
    &=  -4h'(0) \cdot \sum_{i=1}^{D} (\vo_i - \vt_i)(\too_i - \ttt_i)  \tx{ from } \ref{thm:Covariance_structure} \\
    &= -4h'(0) \bigg( \big \langle \vo, \too \big \rangle_2 + \big \langle \vt, \ttt \big \rangle_2 - \big \langle \vo, \ttt \big \rangle_2 - \big \langle \vt, \too \big \rangle_2 \bigg) \tx{ by } \ref{decoupling_techniques}
    \bigskip \\
    &\Cov\bigg( \frac{1}{2} \big((\vo)^\tr \hoo \vo -  (\vt)^\tr \hoo \vt \big), \frac{1}{2} \big((\too)^\tr \hoo \too -  (\ttt)^\tr \hoo \ttt \big) \bigg) \\
    &= \frac{1}{4} \Cov\bigg( \sum_{i,j=1}^{D} \hoo_{i,j} \big( \vo_i \vo_j - \vt_i \vt_j \big), \sum_{k,l=1}^{D} \hoo_{k,l} \big( \too_k \too_l - \ttt_k \ttt_l \big) \bigg) \\
    &= \frac{1}{4} \sum_{i,j=1}^{D} \sum_{k,l=1}^{D} \big( \vo_i \vo_j - \vt_i \vt_j \big) \big( \too_k \too_l - \ttt_k \ttt_l \big) \Cov(\hoo_{i,j}, \hoo_{k,l}) \\
    &= \frac{1}{4} \sum_{i,j=1}^{D} \sum_{k,l=1}^{D} \big( \vo_i \vo_j - \vt_i \vt_j \big) \big( \too_k \too_l - \ttt_k \ttt_l \big) \cdot 8h''(0) \cdot (\dl_{ik}\dl_{jl}+\dl_{il}\dl_{jk}) \tx{ from } Theorem \ref{thm:Covariance_structure} \\
    &= 2h''(0) \cdot \sum_{i,j=1}^{D} \big( \vo_i \vo_j - \vt_i \vt_j \big) \big( \too_i \too_j - \ttt_i \ttt_j \big) \\
    &\quad \ldots + 2h''(0) \cdot \sum_{i,j=1}^{D} \big( \vo_i \vo_j - \vt_i \vt_j \big) \big( \too_j \too_i - \ttt_j \ttt_i \big) \\
    &= 4h''(0) \cdot \sum_{i,j=1}^{D} \big( \vo_i \vo_j - \vt_i \vt_j \big) \big( \too_i \too_j - \ttt_i \ttt_j \big) \\
    &= 4h''(0) \cdot ||(\vo)^\tr \too||_2^2 + ||(\vt)^\tr \ttt||_2^2 - ||(\vo)^\tr \ttt||_2^2 - ||(\vt)^\tr \too||_2^2 \tx{ by } appendix \ref{decoupling_techniques} \\
\end{align*}

\begin{align*}
    &\Cov\big( (\vo)^\tr \hot \vt, (\too)^\tr \hot \ttt \big) \\
    &= \Cov\bigg(\sum_{i,j=1}^{D} \hot_{i,j} \vo_i \vt_j, \sum_{k,l=1}^{D} \hot_{k,l} \too_k \ttt_l \bigg) \\
    &= \sum_{i,j=1}^{D} \sum_{k,l=1}^{D} \vo_i \vt_j \too_k \ttt_l \Cov(\hot_{i,j}, \hot_{k,l}) \\
    &= \sum_{i,j=1}^{D} \sum_{k,l=1}^{D} \vo_i \vt_j \too_k \ttt_l \cdot 8h''(0) \cdot (\dl_{ik}\dl_{jl}-\dl_{il}\dl_{jk}) \tx{ similarly from } \ref{thm:Covariance_structure} \\
    &= 8h''(0) \cdot \sum_{i,j=1}^{D} \vo_i \vt_j \too_i \ttt_j - 8h''(0) \cdot \sum_{i,j=1}^{D} \vo_i \vt_j \too_j \ttt_i \\
    &= 8h''(0) \cdot \sum_{i,j=1}^{D} \vo_i \vt_j \bigg( \too_i \ttt_j - \too_j \ttt_i \bigg) \\
    &= 4h''(0) \cdot \sum_{i,j=1}^{D} \bigg[\vo_i \vt_j \bigg( \too_i \ttt_j - \too_j \ttt_i \bigg) \bigg] \\
    &\quad \ldots +  4h''(0) \cdot \sum_{i,j=1}^{D} \bigg[\vo_j \vt_i \bigg( \too_j \ttt_i - \too_i \ttt_j \bigg) \bigg] \tx{ by renaming indices }\\
    &= 4h''(0) \cdot \sum_{i,j=1}^{D} \bigg(\vo_i \vt_j - \vo_j \vt_i \bigg) \bigg( \too_i \ttt_j - \too_j \ttt_i \bigg) \\
    &= 4h''(0) \big \langle \vo (\vt)^\tr - \vt (\vo)^\tr, \too (\ttt)^\tr - \ttt (\too)^\tr \big \rangle_F \tx{ by } \ref{decoupling_techniques} \\
    &= 4h''(0) \cdot  \bigg( 2 \big \langle \vo, \too \big \rangle_2 \cdot \big \langle \vt, \ttt \big \rangle_2 - 2 \big \langle \vo, \ttt \big \rangle_2 \cdot \big \langle \vt, \too \big \rangle_2 \bigg) 
\end{align*}

\begin{align*}
    &\tx{Denoting }A = \big \langle \vo, \too \big \rangle_2, \quad B = \big \langle \vo, \ttt \big \rangle_2, \quad C = \big \langle \vt, \too \big \rangle_2, \quad D = \big \langle \vt, \ttt \big \rangle_2,\\
    &k_f \big( (\xo, \xt), (\yo, \yt) \big) = -4h'(0) \cdot (A + D - B - C) + 4h''(0) \cdot (A^2 + D^2 - B^2 - C^2 + 2AD - 2BC) \\
    &= -4h'(0) \cdot \big((A + D) - (B + C) \big) + 4h''(0) \cdot \big((A + D)^2 - (B + C)^2 \big) \\
    &= -4h'(0) \cdot \bigg(\big \langle \Delta x, \Delta y \big \rangle_2 - \big \langle \Delta x, (\Delta y)^{flip} \big \rangle_2  \bigg) + 4h''(0) \cdot \bigg(\big \langle \Delta x, \Delta y \big \rangle_2^2 - \big \langle x, (\Delta y)^{flip} \big \rangle_2^2 \bigg) \\
\end{align*}
 
\subsection{Techniques for Decoupling Matrix Inner Products} \label{decoupling_techniques}
Each interacting cross-term corresponds to a polynomial expression that captures coupling among different locations. They can be better understood through the decoupled expressions.
\begin{align*}
    &\sum_{i=1}^{D} \big(\vo_i - \vt_i \big) \big( \too_i - \ttt_i \big) = \big \langle \vo - \vt, \too - \ttt \big \rangle_2 \\
    &= \big \langle \vo, \too \big \rangle_2 + \big \langle \vt, \ttt \big \rangle_2 - \big \langle \vo, \ttt \big \rangle_2 - \big \langle \vt, \too \big \rangle_2 \\
    \bigskip \\
     &\sum_{i,j=1}^{D} \bigg[\big( \vo_i \vo_j - \vt_i \vt_j \big) \big( \too_i \too_j - \ttt_i \ttt_j \big) \bigg] \\
     &= \big \langle \vo (\vo)^\tr - \vt (\vt)^\tr, \too (\too)^\tr - \ttt (\ttt)^\tr \big \rangle_F \\
     &= \big \langle \vo (\vo)^\tr, \too (\too)^\tr \big \rangle_F + \big \langle \vt (\vt)^\tr, \ttt (\ttt)^\tr \big \rangle_F \\
     &\quad \ldots - \big \langle \vo (\vo)^\tr, \ttt (\ttt)^\tr \big \rangle_F - \big \langle \vt (\vt)^\tr, \too (\too)^\tr \big \rangle_F \\
     &= ||(\vo)^\tr \too||_2^2 + ||(\vt)^\tr \ttt||_2^2 - ||(\vo)^\tr \ttt||_2^2 - ||(\vt)^\tr \too||_2^2 \tx{  from the cyclic property of trace } \\
     \bigskip \\
    &\sum_{i,j=1}^{D} \bigg[ \big(\vo_i \vt_j - \vo_j \vt_i \big) \big( \too_i \yt_j - \too_j \ttt_i \big) \bigg] \\ 
    &= \big \langle \vo (\vt)^\tr - \vt (\vo)^\tr, \too (\ttt)^\tr - \ttt (\too)^\tr \big \rangle_F \\
    &= \big \langle \vo (\vt)^\tr, \too (\ttt)^\tr \big \rangle_F + \big \langle \vt (\vo)^\tr, \ttt (\too)^\tr \big \rangle_F - \\
    &\quad \ldots \big \langle \vo (\vt)^\tr, \ttt (\too)^\tr \big \rangle_F - \big \langle \vt (\vo)^\tr, \too (\ttt)^\tr \big \rangle_F \\
    &= 2 \big \langle \vo, \too \big \rangle_2 \cdot \big \langle \vt, \ttt \big \rangle_2 - 2 \big \langle \vo, \ttt \big \rangle_2 \cdot \big \langle \vt, \too \big \rangle_2 \tx{ similarly from the cyclic property of trace } \\
\end{align*}

In the above derivations, we used the trace property in the following way:
\begin{align}
   \big \langle \vo (\vo)^\tr, \too (\too)^\tr \big \rangle_F &= \tx{tr}\big( (\vo (\vo)^\tr)^\tr (\too (\too)^\tr) \big) \nonumber\\
   &= \tx{tr}\big( \vo (\vo)^\tr \too (\too)^\tr \big) \nonumber\\
   &= \tx{tr}\big( (\vo)^\tr \too (\too)^\tr \vo \big) = ||(\vo)^\tr \too||_2^2 \label{first_decoupling}
   \bigskip \\
   \big \langle \vo (\vt)^\tr, \too (\ttt)^\tr \big \rangle_F &= \tx{tr}\big( (\vo (\vt)^\tr)^\tr (\too (\ttt)^\tr) \big) \nonumber\\
    &= \tx{tr}\big( \vt (\vo)^\tr \too (\ttt)^\tr \big) \nonumber \\
    &=  (\vo)^{\tr} \too \cdot tr\big( (\ttt)^\tr \vt \big)\nonumber \\
    &= \big \langle \vo, \too \big \rangle_2 \cdot \big \langle \vt, \ttt \big \rangle_2  \label{second_decoupling}
\end{align}

\newpage
\section{Error Covariance} \label{appendix:Covariance_of_error_process}
\begin{align*}
    &\Cov\big( f(\xo, \xt), g^\tr \big( \too - \ttt \big) \big) \\
    &= \Cov\bigg( f(\xo, \xt), \sum_{i=1}^D g_i (\too_i - \ttt_i) \bigg) \\
    &= \sum_{i=1}^D (\too_i - \ttt_i) \Cov\big( f(\xo, \xt), g_i \big) \\
    &= \sum_{i=1}^D (\too_i - \ttt_i) \Cov\big( f(\xo, \xt), \der^{(1)}_i f(\zz, \zz) \big) \\
    &= \sum_{i=1}^D (\too_i - \ttt_i) \der^{(1)}_i k_f \big( (\xo, \xt), (\zz, \zz) \big) \\
    &= \sum_{i=1}^D (\too_i - \ttt_i) \cdot 4h'\bigg( \|x-z\|_2^2 \bigg) \cdot \big( (\zz_i - \xo_i)  -(\zz_i - \xt_i) \big) \\
    &=- 4h'\big( \|\Delta x\|_2^2 \big) \cdot \sum_{i=1}^D (\too_i - \ttt_i) \cdot \big( \vo_i - \vt_i  \big) \\
    \bigskip \\
    &\Cov\bigg( f(\xo, \xt),  \frac{1}{2} \bigg((\too)^\tr \hoo \too -  (\ttt)^\tr \hoo \ttt \bigg) \bigg) \\
    &= \frac{1}{2} \Cov\bigg( f(\xo, \xt), \sum_{i,j=1}^D \hoo_{i,j} \big( \too_i \too_j - \ttt_i \ttt_j \big)) \\
    &= \frac{1}{2} \sum_{i,j=1}^D \big( \too_i \too_j - \ttt_i \ttt_j \big) \Cov\big( f(\xo, \xt), \hoo_{i,j} \big) \\
    &= \frac{1}{2} \sum_{i,j=1}^D \big( \too_i \too_j - \ttt_i \ttt_j \big) \Cov\big( f(\xo, \xt), \der^{(1,1)}_{(i,j)} f(\zz, \zz) \big) \\
    &= \frac{1}{2} \sum_{i,j=1}^D \big( \too_i \too_j - \ttt_i \ttt_j \big) \der^{(1,1)}_{(i,j)} k_f \big( (\xo, \xt), (\zz, \zz) \big) \\
    &= \frac{1}{2} \sum_{i,j=1}^D \big( \too_i \too_j - \ttt_i \ttt_j \big) \der^{(1,1)}_{(i,j)} 8h''\bigg( \|x-z\|_2^2 \bigg) \cdot \bigg[ (\zz_i - \xo_i)(\zz_j - \xo_j) - (\zz_i - \xt_i)(\zz_i - \xt_j)   \bigg] \\
    &= 4h''\big( \|\Delta x \|_2^2 \big) \cdot \sum_{i,j=1}^D \big( \too_i \too_j - \ttt_i \ttt_j \big) \cdot \big( \vo_i \vt_j - \vt_i \vo_j  \big) 
    \end{align*}
\begin{align*}
    &\Cov\big( f(\xo, \xt), (\too)^\tr \hot \ttt \big) \\
    &= \Cov\bigg( f(\xo, \xt), \sum_{i,j=1}^D \hot_{i,j} \too_i \ttt_j \bigg) \\
    &= \sum_{i,j=1}^D \too_i \ttt_j \Cov\big( f(\xo, \xt), \hot_{i,j} \big) \\
    &= \sum_{i,j=1}^D \too_i \ttt_j \Cov\big( f(\xo, \xt), \der^{(1,2)}_{(i,j)} f(\zz, \zz) \big) \\
    &= \sum_{i,j=1}^D \too_i \ttt_j \der^{(1,2)}_{(i,j)} k_f \big( (\xo, \xt), (\zz, \zz) \big) \\
    &= \sum_{i,j=1}^D \too_i \ttt_j \der^{(1,2)}_{(i,j)} 8h''\bigg( \|x-z\|_2^2 \bigg) \cdot \bigg[ (\zz_i - \xo_i)(\zz_j - \xt_j) - (\zz_i - \xt_i)(\zz_i - \xo_j)   \bigg] \\
    &= 8h''\big( \|\Delta x\|_2^2 \big) \cdot \sum_{i,j=1}^D \too_i \ttt_j \cdot \bigg( \vo_i \vt_j - \vt_i \vo_j  \bigg) \\
    &= 4h''\big( \| \Delta x \|_2^2 \big) \cdot \sum_{i,j=1}^D \big( \too_i \ttt_j - \ttt_i \too_j \big) \cdot \big( \vo_i \vt_j - \vt_i \vo_j  \big) \tx{ by expanding and renaming indices}.
\end{align*}

\begin{align*}
    &\tx{Collecting terms, we have: } \\
    &\Cov\big( f(x), \tilde{f}(y) \big) \\
    &= - 4h'\big( \|\Delta x\|_2^2 \big) \cdot \sum_{i=1}^D (\too_i - \ttt_i) \cdot \big( \vo_i - \vt_i  \big) \\
    &+  4h''\big( \|\Delta x\|_2^2 \big) \cdot \sum_{i,j=1}^D \bigg[ \big( \too_i \too_j - \ttt_i \ttt_j \big) \cdot \big( \vo_i \vt_j - \vt_i \vo_j  \big) \\
    &\quad \ldots + \big( \too_i \too_j - \ttt_i \too_j \big) \cdot \big( \vo_i \vt_j - \vt_i \vo_j  \big) \bigg] \\
    &= -4h'\big( \|\Delta x\|_2^2 \big) \cdot \bigg(\big \langle \Delta x, \Delta y \big \rangle_2 - \big \langle \Delta x, (\Delta y)^{flip} \big \rangle_2  \bigg) + 4h''\bigg( \|\Delta x\|_2^2 \bigg) \cdot \bigg(\big \langle \Delta x, \Delta y \big \rangle_2^2 - \big \langle \Delta x, (\Delta y)^{flip} \big \rangle_2^2 \bigg) \tx{ by } \tx{ appendix }\ref{decoupling_techniques} \\
    % &= -4h'\big( \|x\|_2^2 \big) \cdot \bigg(\big \langle x, y \big \rangle_2 - \big \langle x, y^{flip} \big \rangle_2  \bigg) + 4h''\bigg( \|x\|_2^2 \bigg) \cdot \bigg(\big \langle x, y \big \rangle_2^2 - \big \langle x, y^{flip} \big \rangle_2^2 \bigg) \tx{ by stationarity of } k_f
\end{align*}

\newpage
\section{Distributionally Orthogonal Invariance of Isotropic Gaussian Random Vectors and Matrices Consists of Independent Normal Entries} \label{appendix:orthogonal_invariance}
\begin{lemma} (Orthogonal Invariance of mean-zero Isotropic Gaussian Random Vectors) \label{lem:orthogonal_invariance_isotropic_gaussian_vector}
    Let $v \in \mb{R}^D$ be a Gaussian random vector with mean zero and covariance matrix $\sigma^2 I_D$ for some $\sigma^2 >0$. Then, for any orthogonal transformation $Q: \mb{R}^D \to \mb{R}^D$,
    \begin{align*}
        Q v \overset{d}{=} v
    \end{align*}    
\end{lemma}
\begin{proof}
    $Qv$ also satisfies the definition of a multivariate Gaussian distribution, and is trivially mean-zero. Since they are Gaussian, it suffices to show that they have the same covariance matrices, as follows.
    \begin{align*}
        \Cov (Qv, Qv) &= \E \bigg[ (Qv) (Qv)^\tr  \bigg] \\
        &= \E \bigg[ Qv v^\tr Q^\tr \bigg] \\
        &= Q\E \bigg[ v v^\tr  \bigg]Q^\tr \\
        &= (\sigma^2I_D)QQ^\tr  = \sigma^2I_D
    \end{align*}
\end{proof}

\begin{lemma} \label{lem:orthogonal_invariance_matrix_normal}
    Let $A \in \mb{R}^{D \times D}$ be a random matrix whose entries are independent normal random variables with mean zero and variance $\sigma^2 >0$. Then, for any orthogonal transformation $Q: \mb{R}^D \to \mb{R}^D$.
    Then its distribution is invariant under conjugation by $Q$, i.e.,
    \begin{align*}
        Q^\tr A Q \overset{d}{=} A
    \end{align*}
\end{lemma}
\begin{proof}
    We show that they have the same Probability Density Function.
    Being a random matrix with independent normal entries, the probability density function of $A$ is given as
    \begin{align*}
        f(A) \propto \exp\bigg(-\frac{1}{2\sigma^2} \sum_{i,j = 1}^D A_{i,j}^2 \bigg) = \exp\bigg(-\frac{1}{2\sigma^2} \tx{tr}(A^\tr A) \bigg) 
    \end{align*}
    For the conjugated matrix $Q^\tr A Q$, we have
    \begin{align*}
        f(Q^\tr A Q) &\propto \exp\bigg(-\frac{1}{2\sigma^2} \tx{tr}((Q^\tr A Q)^\tr (Q^\tr A Q)) \bigg) \\
        &= \exp\bigg(-\frac{1}{2\sigma^2} \tx{tr}(Q^\tr A^\tr Q Q^\tr A Q) \bigg) \\
        &= \exp\bigg(-\frac{1}{2\sigma^2} \tx{tr}(A^\tr A) \bigg) 
    \end{align*}
\end{proof}